\documentclass[11pt]{amsart}

\usepackage{amsmath, amssymb, amsthm, amsfonts, bbm, dsfont}
\usepackage{graphicx}

\usepackage{hyperref}
\RequirePackage{doi}

\usepackage{a4wide}

\numberwithin{equation}{section}
\theoremstyle{plain}
\newtheorem{theorem}[subsubsection]{Theorem}
\newtheorem{theoremc}[subsection]{Theorem*}
 \newtheorem{lemma}[subsubsection]{Lemma}
 \newtheorem{prop}[subsubsection]{Proposition}
 \newtheorem{cor}[subsubsection]{Corollary}
 \newtheorem{conj}[subsubsection]{Conjecture}

 \theoremstyle{definition}

\newtheorem{remark}[subsubsection]{Remark}
\newtheorem{exam}[subsubsection]{Example}

\usepackage{hyperref}
\usepackage{pictexwd}
\usepackage{stackengine}
\usepackage{mathtools}
\usepackage{amssymb}
\usepackage{dsfont}
\usepackage{mathrsfs}
\usepackage{mathabx}

\newcommand{\rmL}{\mathrm{L}}

\newcommand{\CC}{\mathbb{C}}
\newcommand{\cc}{\mathbb{C}}

\newcommand{\QQ}{\mathbb{Q}}
\newcommand{\ZZ}{\mathbb{Z}}
\newcommand{\zz}{\mathbb{Z}}
\newcommand{\bbS}{\mathbb{S}}

\newcommand{\rmk}{\mathrm{K}}

\newcommand{\calF}{\mathcal{F}}
\newcommand{\calf}{\mathcal{F}}
\newcommand{\calG}{\mathcal{G}}
\newcommand{\calg}{\mathcal{G}}
\newcommand{\calD}{\mathcal{D}}
\newcommand{\calO}{\mathcal{O}}
\newcommand{\calR}{\mathcal{R}}
\newcommand{\calS}{\mathcal{S}}
\newcommand{\calz}{\mathcal{Z}}
\newcommand{\cb}{\mathcal{B}}
\newcommand{\calB}{\mathcal{B}}
\newcommand{\calL}{\mathcal{L}}

\newcommand{\frg}{\mathfrak{g}}

\newcommand{\frb}{\mathfrak{b}}

\newcommand{\frn}{\mathfrak{n}}

\newcommand{\scz}{\mathscr{Z}}
\newcommand{\scX}{\mathscr{X}}
\newcommand{\scc}{\mathscr{C}}
\newcommand{\scC}{\mathscr{C}}
\newcommand{\scZ}{\mathscr{Z}}

\newcommand{\tilM}{\widetilde{M}}

\newcommand{\tilD}{\widetilde{D}}

\newcommand{\Ind}{\textup{Ind}}
\newcommand{\ind}{\textup{ind}}

\newcommand\Lie{\textup{Lie}\ }

\newcommand\Mod{\textup{Mod}}

\newcommand\Spec{\mathop{Spec}\ }

\newcommand\st{\textup{st}}

\newcommand{\Tr}{\textup{Tr}}
\newcommand{\tr}{\textup{tr}}

\newcommand\Hom{\textup{Hom}}

\newcommand{\ot}{\otimes}
\newcommand{\ti}{\times}

\newcommand\GL{\textup{GL}}
\newcommand\gl{\mathfrak{gl}}

\newcommand{\Ad}{\textup{Ad}}

\newcommand{\bl}{\bullet}

\newcommand{\quash}[1]{}

\newcommand{\pt}{\textup{pt}}

\newcommand{\wt}{\widetilde}

\newcommand{\und}{\underline}

\newcommand{\ChE}{Chevalley-Eilenberg }
\newcommand{\Dr}{\mathrm{Dr}}
\newcommand{\dr}{\mathrm{Dr}}
\newcommand{\rlin}{\mathrm{lin}}

\newcommand{\Br}{\mathfrak{Br}}
\newcommand{\FT}{\mathrm{FT}}

\newcommand{\Coh}{\mathrm{Coh}}

\usepackage{tikz}
\usetikzlibrary{shapes,fit,intersections}
\usetikzlibrary{decorations.markings}
\usetikzlibrary{decorations.pathreplacing}
\usetikzlibrary{patterns}
\usetikzlibrary{matrix,arrows}
\usepgflibrary{decorations.pathmorphing}

\usepackage{tikz-cd}

\newcommand{\Wr}{\overline{W}}
\newcommand{\Hilb}{\textup{Hilb}}
\newcommand{\FHilb}{\mathrm{FHilb}}

\newcommand{\MFs}{\mathrm{MF}^{sc}}
\newcommand{\calXr}{\overline{\mathcal{X}}}
\newcommand{\calX}{\mathcal{X}}
\newcommand{\cx}{\mathcal{X}}
\newcommand{\MF}{\mathrm{MF}}
\newcommand{\calZ}{\mathcal{Z}}
\newcommand{\calC}{\mathcal{C}}
\newcommand{\calY}{\mathcal{Y}}
\newcommand{\Fl}{\mathrm{Fl}}

\newcommand{\frh}{\mathfrak{h}}

\newcommand{\CE}{\mathrm{CE}}
\newcommand{\HHH}{\mathrm{HHH}}

\newcommand{\odel}{\stackon{$\otimes$}{$\scriptstyle\Delta$}}

\newcommand{\HC}{\mathrm{HC}}
\newcommand{\CH}{\mathrm{CH}}

\newcommand{\fgt}{\mathrm{fgt}}

\def\rloc{\mathrm{loc}}

\def\rst{\mathrm{st}}

\def\rper{\mathrm{per}}

\def\gl{\mathfrak{gl}}

\thanks{The work of A.O. was supported in part by  the NSF CAREER grant DMS-1352398, the NSF FRG grant DMS-1760373 grant and Simons Foundation
Fellowship No.561855}

\title{Notes on matrix factorizations and knot homology}

\author{A. Oblomkov}
\address{
A.~Oblomkov\\
Department of Mathematics and Statistics\\
University of Massachusetts at Amherst\\
Lederle Graduate Research Tower\\
710 N. Pleasant Street\\
Amherst, MA 01003 USA
}
\email{oblomkov@math.umass.edu}

\begin{document} 
\maketitle
\begin{abstract}
  These are the notes of the lectures delivered by the author at CIME in June 2018. The main purpose of the
  notes is to provide an overview of the techniques used in the construction of the triply graded link homology.
  The homology is space of global sections of a particular sheaf on the Hilbert scheme of points on the plane.
  Our construction relies on existence on the natural push-forward functor for the equivariant matrix factorizations,
  we explain the subtleties on the construction in these notes. We also outline a proof of the Markov moves for
  our homology as well as some explicit localization formulas for knot homology of
  a large class of links.
  
\end{abstract}

\section{Introduction}
\label{sec:introduction}

The discovery of the knot homology   \cite{Khovanov00}   of the links in the three-sphere, motivated search for the homological
invariants of the three-manifolds. Heegard-Floer homology were discovered soon after Khovanov's seminal work, this homology
categorifies the simplest case of WRT invariants (the invariants at the fourth root of unity).  More general WRT invariants
are beyond of the reach of currently available technique. Thus it is very important to reveal as much structure of the
Khovanov homology as it is possible.

The mathematical construction of WRT invariants   relies on special properties  JW projectors at the root of unity, thus it is natural to
search for the analogues of the projectors in the knot homology theory. If the algebraic variety is endowed with the action of
the torus with the zero-dimensional locus, the algebraic geometry offer a natural decomposition of category of the coherent sheaves
into the mutually orthogonal pieces \cite{HalpernLeistner18}, hence we have a natural analog of the JW projectors. In the paper \cite{OblomkovRozansky18a}
we constructed a map from the braid group to the category of coherent sheaves on  the free Hilbert scheme of points on the plane such that
that Markov moves properties hold for the vector space of the global sections of the sheaf. Thus we have geometric candidate for  JW projectors for such knot homology.

The quest for the geometric interpretation of JW projectors was the main motivation for the author of the notes to develop the
connection between sheaves on the Hilbert scheme of points and knot homology. The localization type formulas were first
encountered by the author in the joint work with Jake Rasmussen and  Vivek Shende \cite{OblomkovRasmussenShende12} where the homology of the torus knots were connected
with the topology of the Hilbert schemes of points on the homogeneous plane singularities (see also \cite{GorskyOblomkovRasmussenShende14}).
However, back in 2012 it was a  total mystery to the author how one would expand the relation in \cite{OblomkovRasmussenShende12}, \cite{GorskyOblomkovRasmussenShende14} beyond the torus knots.

The connection was demystified by Lev Rozansky who was armed with the
physics intuition as well as very deep understanding of already existing knot homology theories. As it turned out the searched after
knot homology has a natural interpretation within the framework of the Kapustin-Saulina-Rozansky topological quantum
field theory for the cotangent bundles
to the Lie algebras  as targets \cite{OblomkovRozansky18b}. A purely mathematical theory underlying the physical predictions
is laid out in the series of our joint papers \cite{OblomkovRozansky16},\cite{OblomkovRozansky17},\cite{OblomkovRozansky17a},
\cite{OblomkovRozansky18},\cite{OblomkovRozansky18a}. To provide an introduction to the technique of these paper
is the main goal of this note.

\subsection{Main result}
\label{sec:main-result}

Let us state a consequence of the results from the papers that requires the minimal amount of new notations. We need some
notations, though. Throughout the paper we  use notation \(D^{per}_G(X)\) for the derived category of the two-periodic
\(G\)-equivariant complexes of coherent sheaves on \(X\), where \(G\) is a group acting on \(X\). For us particularly important  
case of the pair \(X,G\) is \(\Hilb_n(\CC^2),T_{sc}=\cc^*\ti\cc^*\) with scaling action of \(T_{sc}\) on \(\cc^2\). The dual
\(\mathcal{B}\) to the  universal
quotient bundle \(\calB^{\vee}\), \(\calB^\vee|_I=\cc[x,y]/I\) will be used in our construction of the knot homology.

We also use notation
\(\Br_n\) for the braid group on \(n\) strands. For an element \(\beta\in \Br_n\) we can form a link in the three-sphere
\(L(\beta)\) by closing the braid in  the most natural way.

\begin{theorem}\label{thm:HilbC2}\cite{OblomkovRozansky18a} There is a constructive procedure that assigns to a braid
  \(\beta\in \Br_n\) an object \(S_\beta\in D^{per}_{T_{sc}}(\Hilb_n(\cc^2))\) such that
  \begin{enumerate}
    \item \(S_{\beta\cdot \FT}=S_{\beta}\ot \det(\calB)\) where \(\FT\) is the full twist on \(n\) strands
  \item The triply graded vector space \(\HHH(\beta):=H^*(S_\beta\ot \Lambda^\bullet \calB)\) is an
    isotopy invariant of the closure \(L(\beta)\).
  \item  The character of representation of the anti-diagonal torus \(\cc^*_a\subset T_{sc}\) on the spaces
    \(H^*(S_\beta\ot \Lambda^i \calB)\) is the HOMFLYPT polynomial:
    \begin{equation}\label{eq:HOMFLY}
      \sum_i a^i\chi_q(\cc^*_a,H^*(S_\beta\ot \Lambda^i\calB))=\mathrm{HOMFLYPT}(L(\beta)).
    \end{equation}
  \end{enumerate}
  
\end{theorem}

The  constructive procedure in the statement of theorem relies on the theory of matrix factorizations and in this note
we try to present a gentle introduction into the aspects of the theory of matrix factorizations that are necessary
for our theory. The  author of the notes learned theory of matrix factorizations from discussions with Lev Rozansky,
as result the exposition here is quite biased.

The first
construction of the triply-graded categorification of the HOMFLYPT invariant appeared in the seminal work of
Mikhail Khovanov and Lev Rozansky \cite{KhovanovRozansky08b}. It is natural to conjecture that the homology
discussed in these notes coincide with the Khovanov-Rozansky homology.

\subsection{Outline}
\label{sec:outline}

After defining and motivating the category of matrix factorizations in section \ref{exm:xy} we spend some time discussing
the most common type of matrix factorizations, Koszul matrix factorizations in section \ref{sec:kosz-matr-fact}.
The Koszul matrix factorizations are in many regards are analogous to the complete intersection rings and in this section 
we make this analogy more precise by providing a method for constructing a matrix factorization from a complete
intersection (see lemma \ref{lem:ext}).

Next we discuss Knorrer periodicity in section~\ref{sec:knorrer-periodicity} which is the most basic equivalence relation
between the categories of matrix factorizations. After that we explain how one would perform
push-forward and pull-back between the categories of matrix factorizations, see section \ref{sec:functoriality}. Finally,
in the section \ref{sec:equiv-matr-fact} we introduce the equivariant matrix factorizations, in particular we
explain the difference between the strongly and weakly equivariant matrix factorizations, later we only
work with the weakly equivariant matrix factorizations since the weak equivariance allows us to
define the equivariant push-forward. 

In the section \ref{sec:braid-groups-matrix} we explain the key point of our construction, the homomorphism from the braid group
\(\Br_n\) to the category of matrix factorizations. First in the section~\ref{sec:convolution-product} we introduce our main space \(\cx\)
with a potential \(W\) and define a convolution algebra structure
\(\star\) on the category
\(\MF_{\GL_n\ti B^2}(\cx,W)\), here \(B\subset \GL_n\) is a subgroup of upper-triangular matrices.
There is a slightly smaller space \(\bar{\cx}\) with  potential \(\Wr\) such that  Knorrer
periodicity identifies \(\MF_{\GL_n\ti B^2}(\cx,W)\) with \(\MF_{B^2}(\bar{\cx},\Wr)\) and it intertwines the convolution product
\(\star\) with the convolution product \(\bar{\star}\), we provide details in the section~\ref{sec:knorrer-reduction}.
After setting notations for the ordinary and affine  braid groups in section \ref{sec:braid-groups-1}  we state main properties of
the homomorphisms:
\[\Phi: \Br_n\to \MF_{ B^2}(\bar{\cx}^{st},W),\quad \Phi^{aff}:\Br_n^{aff}\to\MF_{B^2}(\bar{\cx},W),\]
the pull-back along \(j_{st}:\cx^{st}\to\cx\) intertwines these homomorphisms. We postpone the details of the construction of
homomorphisms \(\Phi,\Phi^{aff}\) till section~\ref{sec:geom-real-affine}.

In section~\ref{sec:knot-invariants} we explain how one can use the homomorphism \(\Phi\) to construct the triply-graded homology.
The free Hilbert scheme \(\FHilb_n^{free}\) consists of the \(B\)-conjugacy classes \(\FHilb_n^{free}=\wt{\FHilb}_n^{free}/B\)
pairs of matrices with cyclic vector such
that the monomials of the matrices applied to the vector span \(\CC^n\).
There is an embedding of  the \(B\)-cover \(\wt{\FHilb}_n^{free}\) of the free Hilbert scheme into the stable version of our space
\(\wt{\FHilb}_n^{free}\to \bar{\cx}^{st}\) and we define the  homology group:
\[\mathbb{H}^i(\beta):=\mathbb{H}^*((\Phi(\beta)\ot \Lambda^i\calB)^B),\]
where \(\calB\) is the tautological vector bundle over the free Hilbert scheme. It is shown in \cite{OblomkovRozansky16} that the
graded dimension
total sum
\[\HHH(\beta)=\oplus_i H^i(\beta),\]
is a triply graded knot invariant of the closure \(L(\beta)\). We explain in the section~\ref{sec:oj-trace-markov} why this
invariant specializes to the HOMFLYPT invariant after we forget about one of the gradings.
Here \(H^i(\beta)\) is \(\mathbb{H}^{b+i}(\beta)\) with \(b=b(\beta)\) being some specific function of \(\beta\).

The free Hilbert scheme \(\FHilb_n^{free}:=\wt{\FHilb}_n^{free}/B\) is smooth and it contains the usual flag Hilbert
scheme \(\FHilb_n\subset \FHilb_n^{free}\) which is very singular and not even a local complete intersection.
The  relation of our homology with the honest flag Hilbert scheme is the following:
\[\mathbb{S}_\beta=j_e(\Phi(\beta))^B\in D^{per}_{T_{sc}}(\FHilb^{free}_n), \quad \mathrm{supp}\left(\mathcal{H}(\mathbb{S}_\beta)\right)\subset \FHilb_n,\]
where \(\mathcal{H}(\mathbb{S}_\beta)\) is the sheaf of \(\FHilb_n^{free}\) which is the homology of the
two-periodic complex \(\mathbb{S}_\beta\).

The  most non-trivial part of the statement from \cite{OblomkovRozansky16} is the fact that the homology \(\HHH(\beta)\)
do not change under the Markov move that decreases the number of strands in the braid. 
In the section~\ref{sec:markov-relations} we give a sketch of a proof the Markov move invariance, we rely in this section
on the material of section~\ref{sec:geom-real-affine} where the details of the construction of the braid group action are given.

In the section~\ref{sec:sample-computation} we do a simplest computation in the convolution algebra of the category of matrix
factorizations in the case \(n=2\). We show that in \(\MF_{ B^2}(\bar{\cx}^{st},\Wr)\) we have an isomorphism
\begin{equation}\label{eq:blb-sqr}
\calC_\bl\star\calC_\bl\simeq \mathbf{q}^4\calC_\bl\oplus\mathbf{q}^{2}\calC_\bl,\end{equation}
which is the geometric counter-part of the fact that the square of the
non-trivial Soergel bimodule for \(n=2\) is equal to a double of itself \cite{Soergel00}.

Finally, in the section~\ref{sec:chern-funct-local} we define the categorical Chern functor:
\[\CH_{loc}^{st}:\MF_{\GL_n\ti B^2}(\cx^{st},W)\to D^{per}_{T_{sc}}(\Hilb_n(\CC^2)).
\]
We also discuss the properties of the conjugate functor \(\HC_{loc}^{st}\) (see \cite{OblomkovRozansky18c} for
the original construction )which is monoidal.
The  sheaf \(S_\beta\) in the theorem~\ref{thm:HilbC2} is given by:
\[S_\beta=\CH_{loc}^{st}(\Phi(\beta)).\]

The  advantage of the sheaf \(S_\beta\) over \(\mathbb{S}_\beta\) is that it is a \(T_{sc}\)-equivariant periodic complex
of sheaves on the smooth manifold \(\Hilb_n(\CC^2)\) thus we can hope to use \(T_{sc}\)-localization technique for
computation of the knot homology. There are some technical issues with using the localization method directly as we discuss
in \ref{sec:local-form}. We also explain how these technical issue could be circumvented and in particular how one can apply
this technique to compute the homology of the sufficiently positive elements of Jucy-Murphy algebra. This formula
was conjectured in \cite{GorskyNegutRasmussen16}.

\subsection{Other results}
\label{sec:other-results}
We also would like to mention that many relevant aspects of matrix factorizations are not covered in these notes. The  reader
could consult papers the original papers of Orlov for the connections with mirror symmetry \cite{Orlov04} and paper
\cite{Dyckerhoff11} for some further discussion of the foundations of the theory of matrix factorizations and of course
the seminal paper of Khovanov and Rozansky \cite{KhovanovRozansky08b} where the first construction of a triply graded homology of the links
was proposed.
The  constructions in these notes are motivated by the physical theory of Kapustin, Saulina, Rozansky \cite{KapustinSaulinaRozansky09}, the reader is encouraged to read wonderful, basically purely mathematical paper
\cite{KapustinRozansky10} where the role of matrix factorizations in the theory is explained.

Let us also mention that  there is a slightly different perspective on the geometric interpretation
of the knot homology due to Gorsky, Negu\c{t}, Hogencamp and Rasmussen \cite{GorskyNegutRasmussen16}, \cite{GorskyHogancamp17}. Their approach
takes the theory of Soergel bimodules and the corresponding link homology construction \cite{KhovanovRozansky08b}
as a starting point of theory, rather than the categories of matrix factorizations discussed in  these notes.
Finally, let us mention the recent work of Hogencamp and Elias on categorical diagonalization \cite{EliasHogancamp16},\cite{EliasHogancamp17},\cite{EliasHogancamp18} which provides a categorical setting for the
localozation in the category of coherent sheaves.

These notes by no means were intended as a comprehensive survey  of the theory of matrix factorization or
of the theory of knot homology. It is a merely is a slightly extended version of the  three lectures that
the author delivered at 2018 CIME. Thus the author asks for an apology from the colleagues whose
contributions to the fields are not covered in the notes.

{\bf Acknowledgments:} First of all I would like to thank my coathor and friend Lev Rozansky for teaching
everything that is in these notes. All results in these notes are contained in our joint papers.
I also would like to thank Andrei Negu\c{t} and Tina Kanstrup for discussion related to the content
of the notes. I am very grateful to an anonymous referee for many great suggestion on improving the
first draft of the notes.
I am very greatful to CIME foundation for opportunity to teach the course at
the Summer school. The participants of the course provided valuable feed-back on the material.
I was also partially supported by NSF and Simons foundation.

\section{Matrix factorizations}
\label{sec:matr-fact}

In this section we remind some basic facts about matrix factorizations. There are many excellent exposition on matrix factorizations \cite{Eisenbud80},\cite{Orlov04},\cite{Dyckerhoff11} and  we choose not to concentrate on usual matrix factorizations, instead we aim  to define  equivariant matrix
factorizations and subtleties that arise in attempt to define such. We also discuss Koszul matrix factorizations and the (equivariant)
push-forward functor from \cite{OblomkovRozansky16}.

\subsection{Motivation and examples}


Given an affine variety $\mathcal{Z}$ and a function $F$ on it we define \cite{Eisenbud80} the homotopy category $\MF(\mathcal{Z},F)$ of matrix factorizations whose objects are complexes of projective $R=\CC[\calZ]$-modules \(M^0,M^1\)
$M=M^0\oplus M^1$ equipped with the differential \[D=(D^0,D^1)\in \Hom_R(M^0,M^1)\oplus  \Hom_R(M^1,M^0)\] such that $D^2=F$.
Thus $\MF(\mathcal{Z},F)$ is a triangulated category as explained in subsection 3.1 of \cite{Orlov04}. We first discuss
the objects of this category, then discuss various properties of the morphism spaces.

It is convenient to think about a matrix factorization \((M^0\oplus M^1,D)\) as two-periodic curved complex:
\[\dots\xrightarrow{D^1} M^0\xrightarrow{D^0} M^1\xrightarrow{D^1} M^0\xrightarrow{D^0}M^1\xrightarrow{D^1}\dots,\quad D^2=F.\]
Let us look at several
basic examples of matrix factorizations and discuss briefly a motivation for the definition of the matrix factorizations by
Eisenbud \cite{Eisenbud80}.

\begin{exam}\label{exm:x5} \(\calZ=\cc\), \(R=\cc[x]\) and \(F=x^5\). The two-periodic complex
\[\dots\xrightarrow{x^2} \underline{R}\xrightarrow{x^3} R\xrightarrow{x^2} R\xrightarrow{x^3}R\xrightarrow{x^2}\dots\]
is an example of an object in \(\MF(\cc,x^5).\) Here and everywhere below we underline to indicate zeroth homological
degree.
\end{exam}

\begin{exam}\label{exm:xy} \(\calZ=\cc^2\), \(R=\cc[x,y]\), \(F=xy\). The two-periodic complex
  \[\dots\xrightarrow{x} \underline{R}\xrightarrow{y} R\xrightarrow{x} R\xrightarrow{y}R\xrightarrow{x}\dots\]
  is an example of an object in \(\MF(\cc^2,xy)\).
\end{exam}

The  last example has the following geometric interpretation. A module over a quotient ring \(Q=\cc[x,y]/(xy)\), in general,
does not have  a finite free resolution. In particular, \(M=\cc[x]=Q/(y)\) is a module over \(Q\) with an infinite
free resolution:
\[0\leftarrow M\xleftarrow{y}Q\xleftarrow{x}Q\xleftarrow{y}Q\xleftarrow{x}\dots.\]
This resolution has a two-periodic (half-infinite) tail which is a reduction of the matrix factorization from the example~\ref{exm:xy}.
As explained in \cite{Eisenbud80} this phenomenon is more general.



We felt obliged to mention these results on matrix factorizations to honor the origins of the subjects. For further development of Eisenbud
theory the reader is encouraged to look at \cite{Eisenbud80} as well as \cite{Orlov04},\cite{Orlov09},\cite{Orlov12} where the connection
with B-model theory is developed. However, the hypersurfaces defined by the potentials from \cite{OblomkovRozansky16} do not
have a clear geometric interpretation and it is unclear to us how to make use of Eisenbud's theory in our case. Instead, more elementary
homological aspect of the matrix factorizations is important to us. Roughly stated, the very important observation is
that all important homological information about the category of matrix factorizations is contained in a neighborhood of the critical
locus. We explain more rigorous statement below.

It is a good place to define morphisms in the category of matrix factorizations. Suppose we have two objects
\(\calf_1=(M_1,D_1),\calf_2=(M_2,D_2)\in \MF(\calZ,F)\) then we define:
\[\und{\Hom}(\calf_1,\calf_2):=\{\Psi\in \Hom_R(M_1,M_2)|\Psi\circ D_1=D_2\circ \Psi\}.\]

Since the modules \(M_i\) are \(\zz_2\)-graded we have a decomposition \[\und{\Hom}(M_1,M_2)=\oplus_{i\in\zz_2}\und{\Hom}^i(M_1,M_2)\]
where \(\und{\Hom}^i(M_1,M_2)\subset \Hom_R^i(M_1,M_2):=\Hom_R(M_1^0,M_2^i)\oplus \Hom_R(M_1^1,M_2^{i+1}).\)

We say that an element \(\Psi\in \und{\Hom}^0(\calF_1,\calF_2)\) is homotopic to zero:
\(\Psi\sim 0\) if there is \(h\in \Hom^1(M_1,M_2)\) such that \(\Psi=h\circ D_1+ D_2\circ h\). Finally, we define the space of morphisms
as a set of equvalence classes with respect to the homotopy equivalence:
\[\Hom(\calf_1,\calf_2):=\und{\Hom}^0(\calf_1,\calf_2)/\sim\]

Now that we defined the objects and morphisms between the objects we can state Orlov's theorem

\begin{theorem}\cite{Orlov04} \(\MF(\calz,F)\) has a structure of the triangulated category.
\end{theorem}

  To complete our discussion of the homological properties of category of matrix factorizations with respect to their critical
  locus let us observe that an element \(f\in R\) naturally gives an element of \(\Hom(\calF,\calF)\). For simplicity let us
  also assume that \(\calZ\subset \cc^m\). Then we have a well-defined ideal \(I_{crit}\subset R\) generated by \(\frac{\partial F}{\partial x_i}\)
  \(i=1,\dots,m\) and \(x_i\) are coordinates on \(\cc^m\).

 \begin{prop}\label{prop:crit}
    For any \(\calf\in\MF(\calz,F)\) and \(f\in I_{crit}\) we have:
    \[\und{\Hom}^0(\calF,\calf)\ni f\sim 0.\]
  \end{prop}
\begin{proof}
    If it is enough to show the statement for \(f=\frac{\partial F}{\partial x_i}\). Thus the statement follows since:
    \[\frac{\partial F}{\partial x_i}=\frac{\partial D}{\partial x_i}D+D\frac{\partial D}{\partial x_i},\]
    and \(\frac{\partial D}{\partial x_i}\) provides the needed homotopy.
  \end{proof}

  The  last proposition  implies that category of matrix factorizations is model for the coherent sheaves on possibly singular
  critical locus of the potential \(F\). When the potential is linear in by some set of variables than there is an equivalence
  between with the DG category of the critical locus (see section \ref{sec:line-kosz-dual} for more discussion). Another manifestation of this
  principle is the shrinking lemma, see lemma~\ref{lem:shrink}   below.
  
\subsection{Koszul matrix factorizations}
\label{sec:kosz-matr-fact}

The  matrix factorizations from examples \ref{exm:xy} and \ref{exm:x5} are examples of so called Koszul matrix factorizations which we
discuss in this subsection. Suppose we have a presentation of the potential as sum \(F=\sum_{i=1}^n a_ib_i\). Then we define Koszul
matrix factorization \(\rmk[\vec{a},\vec{b}]\in \MF(\calz,F)\) as
\[\rmk[\vec{a},\vec{b}]:=(\Lambda^\bullet V,D),\quad D=\sum_i a_i\theta_i+b_i\frac{\partial}{\partial \theta_i},\]
where \(V=\langle \theta_1,\dots,\theta_n\rangle\). The  examples \ref{exm:x5}, \ref{exm:xy} are \(\rmk[x^2,x^3]\) and \(\rmk[x,y]\), respectively. 

The Koszul matrix factorizations are tensor products of the simplest Koszul matrix factorizations. Indeed, given two
matrix factorizations \(\calf_1\in \MF(\calz,F_1)\),  \(\calf_2\in \MF(\calz,F_2)\) the tensor product
\(\calf_1\otimes\calf_2\in \MF(\calz,F_1+F_2)\) is the matrix factorization \((M_1\otimes M_2,D_1\otimes 1+1\otimes D_2)\).
Thus we have \(\rmk[\vec{a},\vec{b}]=\otimes_{i=1}^n\rmk[a_i,b_i]\).

An object of the category of matrix  factorizations with the zero potential is a two-periodic complex of coherent sheaves.
We denote by \(D^{per}(\calz)\) the derived category of the two-periodic complexes of coherent sheaves. Given
two matrix factorizations \(\calf_1\in \MF(\calz,F)\),  \(\calf_2\in \MF(\calz,-F)\) their tensor product is
an element of \(D^{per}(\calz)\) and proposition~\ref{prop:crit} implies:

\begin{cor} For  \(\calf_1\in \MF(\calz,F)\),  \(\calf_2\in \MF(\calz,-F)\) homology of
  the two-periodic complex \(\calf_1\otimes \calf_2 \) are supported on the zero locus of
  \(I_{crit}\).
  \end{cor}

  Now let us discuss a method for constructing interesting Koszul matrix factorizations. Let us first recall some basic
  properties of the usual Koszul complexes. The sequence \(f_1,\dots,f_m\in R\) is called {\it regular} if
  \(f_i\) is not a zero-divisor in the quotient \(R/(f_1,\dots,f_{i-1} )\) for \(i=1,\dots,n\). It is known that the regularity
  does not depend on the order of the elements. There is an equivalent way to define regularity with the help of
  Koszul complexes. The  Koszul complex of \(\vec{f}\) is:
  \[\rmk[\vec{f}]=(\Lambda^\bullet V,D),\quad D=\sum_if_i\frac{\partial}{\partial \theta_i}.\]
  \begin{prop} The sequence \((f_1,\dots,f_m)\) is regular if and only if:
    \[H^i(\rmk[\vec{f}])=0, \quad i>0,\quad,\quad H^0(\rmk(\vec{f}))=R/(f_1,\dots,f_m).\]
  \end{prop}

  Given a finite complex of \((C_\bl,d)\)  of free \(R\)-modules we denote by \([C_\bl]_{per}\) the two-periodic folding of the
  complex. It is an element of \(\MF(\calz,0)\). Suppose \(F\in (f_1,\dots,f_m)\) and
  the sequence \(\vec{f}\) is regular. Then the lemma  below shows that there is an essentially unique way to deform
  the complex \([\rmk[\vec{f}]]_{per}\) to an element of \(\MF(\calz,F)\). We outline a proof of the lemma to
  demonstrate the key deformation theory technique that is used in many constructions of \cite{OblomkovRozansky16}.

  \begin{lemma} \label{lem:ext}Suppose \(F\in (f_1,\dots,f_m)\) and
    the sequence \(\vec{f}\) is regular.Then the Koszul complex
    \[C_\bl=\rmk[\vec{f}]=\{C_0\xleftarrow{d_1^+}C_1\xleftarrow{d_2^+}\dots \xleftarrow{d_m^+}C_m\}\]
    could be completed with the opposite differentials \(d_i^-:C_\bl\to C_{\bl+2i-1}\), \(i>0\) such that
    \[(C_\bl,d^++d^-)\in \MF(\calz,F).\]
  \end{lemma}
  \begin{proof}
    We will construct the differentials \(d_i\) iteratively. Since  the sequence is regular we have
    a homotopy equivalence:
    \begin{equation}\label{eq:hmtp}(C_\bl,d^+)\sim Q=R/(f_1,\dots,f_m).\end{equation}
    Let us also introduce notation for the graded pieces of the space of homomorphisms:
    \[\Hom^i(C_\bl,C_\bl)=\oplus_j \Hom(C_{j},C_{-i+j}).\]  The element \(F\) is an endomorphism of \((C_\bl,d^+)\) and because of
    \eqref{eq:hmtp} it is homotopic to zero by the lemma assumptions. Thus there is a homotopy \(h^{(-1)}\in \Hom^{-1}(C_\bl,C_\bl)\) such that
    \(F=h^{(1)}\circ d^++d^+\circ h^{(1)}\).  Let us set \(D^{(1)}=d^++h^{(1)}\).

    The  differential \(D^{(1)}\) is the first order approximation for our desired extension. It is not differential of a matrix
    factorization if \(n>1\) since:
    \[(D^{(1)})^2=F+(h^{(1)})^2.\]
    However the correction term \((h^{(1)})^2\) is actually an element of \(\Hom^{-2}_{d^+}(C_\bl,C_\bl)\), that is it commutes with
    the differential \(d^+\):
    \[d^+\circ h^{(1)}\circ h^{(1)}=F h^{(1)}-h^{(1)}\circ d^+\circ h^{(1)}=F h^{(1)}+h^{(1)}\circ h^{(1)}\circ d^+-h^{(1)}F=h^{(1)}\circ h^{(1)}\circ d^+.\]
    Thus again by \eqref{eq:hmtp} there is a homotopy \(h^{(3)}\in \Hom^{-3}(C_\bl,C_\bl)\) such that \(h^{(1)}=d_+\circ h^{(3)}+ h^{(3)}\circ d_+\).
    We define the next approximation to the needed differential \(D^{(3)}=d^++h^{(1)}+h^{(3)}\). Again \(D^{(3)}\) is not a differential
    of a matrix factorization if \(n>3\):
    \begin{multline*}
     (D^{(1)}+h^{(3)})^2=F+(h^{(1)})^2+D^{(1)}\circ h^{(3)}+ h^{(3)}\circ D^{(1)}+(h^{(3)})^2=F+h^{(1)}\circ h^{(3)}+h^{(1)}\circ h^{(3)}+(h^{(3)})^2.
    \end{multline*}
    The correction term belongs to \(\Hom^{<-3}(C_\bl,C_\bl)\) and the degree four piece of this term is
    \( h^{(1)}\circ h^{(3)}+h^{(3)}\circ h^{(1)}\).  Let us check that \( h^{(1)}\circ h^{(3)}+h^{(3)}\circ h^{(1)}\in \Hom^{-4}_{d^+}(C_\bl,C_\bl)\):
    \begin{gather*}
      d^+\circ h^{(1)}\circ h^{(3)}=F h^{(3)}-h^{(1)}\circ d^+\circ h^{(3)}=F h^{(3)} -(h^{(1)})^3-h^{(1)}\circ h^{(3)}\circ d^+,\\
      d^+\circ h^{(3)}\circ h^{(1)}= h^{(1)}\circ h^{(1)}\circ h^{(1)}-h^{(3)}\circ d^+\circ h^{(1)}=(h^{(1)})^3-h^{(3)}F-h^{(3)}\circ h^{(1)}\circ d^+.
    \end{gather*}
    By the same argument as before homomorphism     \( h^{(1)}\circ h^{(3)}+h^{(3)}\circ h^{(1)}\) is homotopic to zero and let denote
    by \(h^{(5)}\in \Hom^{-5}(C_\bl,C_\bl)\). The next approximation for our differential is \(D^{(5)}=d^++h^{(1)}+h^{(3)}+h^{(5)}\) and
    \[(D^{(5)})^2-F\in \Hom^{<-5}(C_\bl,C_\bl).\]
    Similar method could be applied to show that correction term of degree six is homotopic to zero and thus we have the next
    order correction. Clearly, this iterative procedure terminates since our complex is of finite length. More formal proof of
    the lemma is given in lemma 2.1 in \cite{OblomkovRozansky16}.
  \end{proof}

  \begin{remark}
    The only assumption on the complex \((C_\bl,d^+)\) that we used is that
    \begin{equation}\label{eq:hom0}
      \Hom^{<0}_{d^+}(C_\bl,C_\bl)\sim 0.
      \end{equation} Thus we can strengthen our
    lemma a little bit by replacing regularity Koszul complex by the condition \eqref{eq:hom0}
  \end{remark}

  It is natural to ask how canonical is the matrix factorization \((C_\bl,d^++d^-)\) constructed in the previous lemma.
  Clearly, our method relies on a existence of various homotopies which are not unique. However, one can show that
  outcome of the iterative procedure in the proof is unique up to an isomorphism. We invite reader to try to apply the
  iterative method of the previous lemma to show lemma below, a formal proof could be found in lemma 3.7 in \cite{OblomkovRozansky16}.

  \begin{lemma}
    Let \((C_\bl,d^+)\) be a complex of free modules with non-trivial terms in degrees from \(0\) to \(l\ge 0\) such that
    \(\Hom_{d^+}^{<0}(C_\bl,C_\bl)\sim 0\). Suppose we have two matrix factorizations
    \[\calf=(C_\bl,d^++d^-),\tilde{\calf}=(C_\bl,d^++\tilde{d}^-)\in \MF(\calz,F),\]
    where \(d^-=\sum_{i\ge 0} d^-_i\), \(\tilde{d}^-=\sum_{i\ge 0} \tilde{d}^-_i\),
    \(d^-_i,\tilde{d}^-_i\in \Hom^{-2i-1}(C_\bl,C_\bl)\) and \(F\sim 0\) as endomorphism of
    \( (C_\bl,d^+)\). Then there is \(\Psi=1+\sum_{i>0}\Psi_i\), \(\Psi_i\in \Hom^{-i}(C_\bl,C_\bl)\) such that
    \[\Psi\circ (d^++d^-)\circ \Psi^{-1}=d^++\tilde{d}^-.\]
  \end{lemma}

  Because of the previous lemma we will use notation \(\rmk^F(f_1,\dots,f_m)\in \MF(\calz,F)\) for a matrix factorization
  from the lemma~\ref{lem:ext}.

\subsection{Knorrer periodicity}
\label{sec:knorrer-periodicity}

The  critical locus of the potential \(F=xy\) is a point \(x=y=0\) so according to our principle we expect that
the category of matrix factorizations with the potential \(xy\) is equivalent to the category of matrix factorizations
on the point. It is indeed the case and the equivalence is known under the name {\it Knorrer periodicity} and we explain
the details below.

Let us denote the Koszul matrix factorization \(\rmk[x,y]\in \MF(\cc^2,xy)\) by \(\rmk\). Then there is an exact
functor between  triangulated categories:
\[\Phi:\MF(\pt,0)\to \MF(\cc^2,xy),\quad (M,D)\mapsto (M\ot\cc[x,y],D)\ot\rmk.\]
The functor in the inverse direction is the restriction functor:
\[\Psi: \MF(\cc^2,xy)\to \MF(\pt,0),\quad (M,D)\mapsto (M|_{x=0},D|_{x=0}).\]

These functors are mutually inverse. Indeed, to show that \(\Psi\circ \Phi=1\) we observe that
\(\rmk|_{x=0}=[\cc[y]\xleftarrow{y}\cc[y]]\) which is a sky-scarper at \(y=0\). We leave it as an exercise
to a reader to show \(\Phi\circ\Psi=1\).

More generally, if \(\calz=\calz_0\ti \cc^2_{x,y}\) and \(F_0\in \cc[\calz_0]\) then there is a functor:
\(\Phi: \MF(\calz_0,F_0)\to \MF(\calz,F_0+xy)\) given by tensoring with the Koszul complex \(\rmk[x,y]\).

\begin{theorem}\cite{Orlov04} The functor \(\Phi\) is an equivalence of triangulated categories.
\end{theorem}

\subsection{Functoriality}
\label{sec:functoriality}
Now we will use previously developed technique to define the push-forward functor for matrix factorizations.
First we discuss a construction of the push-forward for  embedding map:
\(j:\calz_0\hookrightarrow \calz\) where \(\calz=\Spec(S)\) and \(\calz_0=\Spec(R)\), \(R=S/I\).

\begin{theorem}\cite{OblomkovRozansky16} \label{thm:push}Suppose we have \(F\in S\), \(F_0=j^*(F)\) and
  \(I=(f_1,\dots,f_m)\) where \(f_i\) form a regular sequence. Then there is well-defined
  functor of triangulated categories:
  \[j_*: \MF(\calz_0,F_0)\to \MF(\calz,F)\]
\end{theorem}

Given an element \(\calf=(M,D)\in \MF(\calz_0,F_0)\) 
let us explain the construction of the element \(j_*(\calf)=\widetilde{F}\in \MF(\calz,F)\).
Since \(M=R^n\) for some \(n\) we can lift it to the module \(\tilM=S^n\) as well as the differential
to a \(\zz_2\)-graded endomorphism \(\tilD\in \Hom_S(S^n,S^n)\), \(\tilD|_{\calz_0}=D\).
Since \(f_i\) form a regular sequence we can form Koszul complex \(\rmk(f_1,\dots,f_m)=(\Lambda^\bl \cc^n\ot S,d_K)\)
which is a resolution of \(S\)-module \(R\). The technique similar to the method of lemma~\ref{lem:ext}
yields

\begin{lemma}\cite{OblomkovRozansky16} There are \(d_{ij}:\tilM\ot \Lambda^\bl \cc^n\ot S\), \(i-j\in 1+\zz_{\ge 0}\) such that
  \[\widetilde{\calf}=(\Lambda^\bl\cc^n\ot S,d_K+\tilD+d^-)\in \MF(\calz,F)\]
  and the element \(\widetilde{\calf}\) is unique up  to isomorphism.
\end{lemma}

To complete proof of the theorem~\ref{thm:push} we need to show that the construction of \(j^*\) extends to the
spaces of the morphisms between the objects and to the space of homotopies between the morpshism, it is shown
in lemma 3.7 of \cite{OblomkovRozansky16} and we refer interested reader there for the technical details.

Unlike push-forward the pull-back functor is rather elementary. Suppose we have \(f:\calz\to\calz_0\) a morphism
of affine varieties and \(F=f^*(F_0)\), \(F_0\in \cc[\calz_0]\).  Since pull-back of a free module is free, we have
a well-defined functor:
\[f^*: \MF(\calz_0,F_0)\to \MF(\calz,F).\]

Finally, let us remark that the above defined pull-back and push-forward functors satisfy the base change relation
for commuting squares of maps.

\subsection{Equivariant matrix factorizations}
\label{sec:equiv-matr-fact}

Matrix factorization is a natural object attached to a function on the affine manifold. However limiting
yourself to only affine manifold is  frustrating, so one would like to
develop a theory of matrix factorizations on quasi-projective manifolds. There are some
proposals in the literature for such theory, see for example \cite{PolishchukVaintrob11}.

In our work \cite{OblomkovRozansky16} we chose
an approach that is probably more limited than the one from \cite{PolishchukVaintrob11} but has an advantage of being
computation friendly. So in \cite{OblomkovRozansky16} to explore matrix factorizations on the manifolds
that are group quotients of the affine manifolds, we develop theory of equivariant matrix factorizations.
In this section  we motivate our definitions and outline the ingredients of the construction from \cite{OblomkovRozansky16}.

Suppose the affine manifold \(\calz\) has an action of an algebraic group \(H\) and \(F\in \cc[\calz]^H\). Then one can say
the matrix factorization 
\(\calf=(M,D)\in \MF(\calz,F)\) is {\it strongly \(G\)-equivariant} if \(
M\) is endowed with \(H\)-representation structure and the differential \(D\) is \(H\)-equivariant. Let us denote
the set of strongly equivariant matrix factorizations by \(\MF_H^{str}(\calz,F)\). By requiring the morphism between the
objects and the homotopies between the morphisms to be \(H\)-equivariant we can provide \(\MF_H^{str}(\calz,F)\) with the
structure of the triangulated category.

However, the notion of strong equivariance turns out to be too restrictive. Indeed, one of the key tools in our arsenal
is the extension lemma~\ref{lem:ext} together with the push-forward functor. So we would like to have analog of
lemma~\ref{lem:ext} in the equivariant setting, for the \(H\)-equivariant ideal \(I=(f_1,\dots,f_m)\) with \(f_i\) forming a
regular sequence. Unfortunately, the proof of the lemma fails in the strongly equivariant setting because we can not guarantee
that the homotopies in the iterative construction of proof are equivariant. If \(H\) is reductive, we can save the proof
by averaging along the maximal compact subgroup of \(H\). But for non-reductive group we need a weaker notion of
equivariance that relies on the \ChE complex explained below.

Let $\frh$ be the Lie algebra of  \(H\). Chevalley-Eilenberg complex
 $\CE_\frh$ is the complex $(V_\bullet(\frh),d)$ with $V_p(\frh)=U(\frh)\otimes_\CC\Lambda^p \frh$ and differential $d_{ce}=d_1+d_2$ where:
 \def\dtheta{d}
 $$ d_1(u\otimes x_1\wedge\dots \wedge x_p)=\sum_{i=1}^p (-1)^{i+1} ux_i\otimes x_1\wedge\dots \wedge \hat{x}_i\wedge\dots\wedge x_p,$$
 $$ d_2(u\otimes x_1\wedge\dots \wedge x_p)=\sum_{i<j} (-1)^{i+j} u\otimes [x_i,x_j]\wedge x_1\wedge\dots \wedge \hat{x}_i\wedge\dots\wedge \hat{x}_j\wedge\dots \wedge x_p,$$

 Let us denote by $\Delta$ the standard map $\frh\to \frh\otimes \frh$ defined by $x\mapsto x\otimes 1+1\otimes x$.
 Suppose $V$ and $W$ are modules over the Lie algebra $\frh$ then we use notation
 $V\odel W$ for  the $\frh$-module which is isomorphic to $V\otimes W$ as a vector space, the $\frh$-module structure being defined by  $\Delta$. Respectively, for a given $\frh$-equivariant matrix factorization $\calF=(M,D)$ we denote by $\CE_{\frh}\odel \calF$
 the $\frh$-equivariant matrix factorization $(CE_\frh\odel\calF, D+d_{ce})$. The $\frh$-equivariant structure on $\CE_{\frh}\odel \calF$ originates from the
 left action of $U(\frh)$ that commutes with right action on $U(\frh)$ used in the construction of $\CE_\frh$.

 A slight modification of the standard fact that $\CE_\frh$ is the resolution of the trivial module implies that \(\CE_\frh\stackon{$\otimes$}{$\scriptstyle\Delta$} M\) is a free resolution of the
$\frh$-module $M$.

Now we about to define a new category whose objects we refer to as {\it weakly equivariant matrix factorizations}. The objects of this category \(\MF_{\frh}(\calZ,W)\) are  triples:
\[\mathcal{F}=(M,D,\partial),\quad (M,D)\in\MF(\calZ,W) \]
where $M=M^0\oplus M^1$ and $M^i=\CC[\calZ]\otimes V^i$, $V^i \in \Mod_{\frh}$,
$\partial\in \oplus_{i>j} \Hom_{\CC[\calZ]}(\Lambda^i\frh\otimes M, \Lambda^j\frh\otimes M)$ and $D$ is an odd endomorphism
$D\in \Hom_{\CC[\calZ]}(M,M)$ such that
$$D^2=F,\quad  D_{tot}^2=F,\quad D_{tot}=D+d_{ce}+\partial,$$
where the total differential $D_{tot}$ is an endomorphism of $\CE_\frh\odel M$, that commutes with the $U(\frh)$-action.


Note that we do not impose the equivariance condition on the differential $D$ in our definition of matrix factorizations. On the other hand, if $\calF=(M,D)\in \MF^{str}(\calZ,F)$ is a matrix factorization with
$D$ that commutes with $\frh$-action on $M$ then $(M,D,0)\in \MF_\frh(\calZ,F)$.

There is a  forgetful map for the objects of the categories  $\mathrm{Ob}(\MF_\frh(\calZ,F))\to
\mathrm{Ob}(\MF(\calZ,F)$ that forgets about the correction differentials:
$$\calF=(M,D,\partial)\mapsto \calF^\sharp:=(M,D).$$

Given two $\frh$-equivariant matrix factorizations $\calF=(M,D,\partial)$ and $\tilde{\calF}=(\tilde{M},\tilde{D},\tilde{\partial})$ the space of morphisms $\Hom(\calF,\tilde{\calF})$ consists of
homotopy equivalence classes of elements $\Psi\in \Hom_{\CC[\calZ]}(\CE_\frh\odel M, \CE_\frh\odel \tilde{M})$ such that $\Psi\circ D_{tot}=\tilde{D}_{tot}\circ \Psi$ and $\Psi$ commutes with
$U(\frh)$-action on $\CE_\frh\odel M$. Two maps $\Psi,\Psi'\in \Hom(\calF,\tilde{\calF})$ are homotopy equivalent if
there is \[ h\in  \Hom_{\CC[\calZ]}(\CE_\frh\odel M,\CE_\frh\odel\tilde{M})\] such that $\Psi-\Psi'=\tilde{D}_{tot}\circ h- h\circ D_{tot}$ and $h$ commutes with $U(h)$-action on  $\CE_\frh\odel M$.

 Given two $\frh$-equivariant matrix factorizations $\calF=(M,D,\partial)\in \MF_\frh(\calZ,F)$ and $\tilde{\calF}=(\tilde{M},\tilde{D},\tilde{\partial})\in \MF_\frh(\calZ,\tilde{F})$
 we define $\calF\otimes\tilde{\calF}\in \MF_\frh(\calZ,F+\tilde{F})$ as the equivariant matrix factorization $(M\otimes \tilde{M},D+\tilde{D},\partial+\tilde{\partial})$.



We define an embedding-related push-forward in the case when the subvariety $\calZ_0\xhookrightarrow{j}\calZ$
is the common zero of an ideal $I=(f_1,\dots,f_n)$ such that the functions $f_i\in\CC[\calZ]$ form a regular sequence. We assume that the Lie algebra $\frh$ acts on $\calZ$ and $I$ is $\frh$-invariant. In section~3 of \cite{OblomkovRozansky16} we use technique
similar to the proof of lemma~\ref{lem:ext} to show that there is a well-defined functor:
\[
j_*\colon \MF_{\frh}(\calZ_0,W|_{\calZ_0})\longrightarrow
\MF_{\frh}(\calZ,W),
\]
for any $\frh$-invariant element $W\in\CC[\calZ]^\frh$.

For our construction of the convolution algebras we also need to define the equivariant  push-forward along a projection.
 Suppose \(\calZ=\mathcal{X}\times\mathcal{Y}\), both \(\calZ \) and \(\mathcal{X}\) have \(\frh\)-action and
the projection \(\pi:\mathcal{Z}\rightarrow\mathcal{X}\) is \(\frh\)-equivariant. Then
for any $\frh$ invariant element $w\in\CC[\calX]^\frh$ there is a functor
\(\pi_{*}\colon \MF_{\frh}(\calZ, \pi^*(w))\rightarrow \MF_{\frh}(\mathcal{X},w)
\)
which simply forgets the action of $\CC[\calY]$.



Finally, let us discuss the quotient map. The complex \(\CE_\frh\) is a resolution of the trivial \(\frh\)-module by free modules. Thus the correct derived
version of taking \(\frh\)-invariant part of the matrix factorization \(\mathcal{F}=(M,D,\partial)\in\MF_\frh(\calZ,W)\), \(W\in\CC[\calZ]^\frh\) is
\[\CE_\frh(\mathcal{F}):=(\CE_\frh(M),D+d_{ce}+\partial)\in\MF(\calZ/H,W),\]
where \(\calZ/H:=\mathrm{Spec}(\CC[\calZ]^\frh )\) and use the general definition for an \(\frh\)-module \(V\):
\[\CE_\frh(V):=\Hom_\frh(\CE_\frh,\CE_\frh\odel V).\]

\section{Braid groups and matrix factorizations}
\label{sec:braid-groups-matrix}

In this section we explain a construction for an action of the finite and affine braid groups on the particular
categories of the matrix factorizations from \cite{OblomkovRozansky16}. First we explain a construction for the convolution algebra on our
categories of matrix factorizations. Then we explain a categorization of the homomorphism from the affine
braid group to the finite braid group from \cite{OblomkovRozansky17}.

\subsection{Convolution product}
\label{sec:convolution-product}

Let us first motivate the definition of the space that host our categories of matrix factorizations. Somewhat abusing notations we
introduce space \(\sqrt{\cx}=\gl_n\ti\GL_n\ti\frn\) where \(\frn\) stands for the Lie algebra of strictly upper-triangular matrices.
We omit the sub-index since the size of the  matrices is clear from the context, we also use \(G\) and
\(\frg\) for \(\GL_n\) and \(\gl_n\) in this situation.

The space \(\sqrt{\cx}\) has the action
of the group of upper-triangular matrices \(B\) and \(G\):
\[(h,b)\cdot (X,g,Y)=(\Ad_h(X),hgb,\Ad_b^{-1}Y), \quad (h,b)\in G\times B.\]

The  flag variety \(\Fl\) is a quotient \(G/B\)  since every full flag can be
moved into the standard flag by \(G\)-action and \(B\) is the stabilizer group of
the standard flag . The  group \(B\) acts on the tangent space to \(\Fl\) at the point of
standard flag and as \(B\)-module the tangent space is equal \(\frn\).
Thus the \(B\)-quotient of \(\sqrt{\cx}\) is naturally isomorphic to the cotangent bundle to the
flag variety: \[\sqrt{\cx}/B=\frg\ti T^*\Fl\]
Thus \(G\)-action on \(\sqrt{\cx}\) induces the \(G\)-action on \(\frg\ti T^*\Fl\).

The space \(T^*\Fl\) is symplectic and the \(G\)-action respects preserves the symplectic form. Thus there is moment map
\(\mu: T^*\Fl\to \frg^*\). The trace identifies \(\frg\) with \(\frg^*\) and we can think of the moment map as \(\frg\)-linear \(B\)-invariant  function:
\[\mu: \sqrt{\cx}\to \cc,\quad \mu(X,g,Y)=\Tr(X\Ad_gY).\]

Now we can define our main space where the convolution algebra dwells. The space \(\sqrt{\cx}\) has \(B\)-invariant projection
to the first factor and our main space is the fibered product:
\[\cx:=\sqrt{\cx}\ti_{\frg}\sqrt{\cx}=\frg\ti G\ti \frn\ti G\ti \frn.\]

The space \(\cx\) has a action of \(G\times B^2\) that is induced from the \(G\ti B\) action on \(\sqrt{\cx}\),
respectively the projections \(p_1,p_2\) on two copies of \(\sqrt{\cx}\) are \(G\times B^2\)-equivariant.
The  group \(B\) is a semi-direct product \(B=T\ltimes U\) of the torus \(T\) and the group of upper-triangular
matrices \(U\).

We define our main
category to be:
\[\MF_n:=\MF_{G\times B^2}(\cx,W),\quad W=p_1^*(\mu)-p_2^*(\mu),\]
where we require the weak \(U^2\)-equivariance and strong \(G\ti T^2\)-equivariance in our category.
The strong \(G\times T^2\)-equivariance implies that all differentials in the complexes are \(G\times T^2\)-invariant.
We can combine strong \(G\ti T^2\)-equivariance with the weak \(U^2\)-equivariance since the \ChE complex for \(U^2\)
is \(G\ti T^2\)-invariant.

There is an action of \(T_{sc}=\cc^*\ti\cc^*=\cc^*_a\ti \cc^*_t\) on the space \(\sqrt{\cx}\) and the induced action on  \(\cx\):
\[(\lambda,\mu)\cdot (X,g,Y)=(\lambda^2\cdot X,g,\lambda^{-2}\mu^{2}Y).\]
The potential \(W\) is not \(T_{sc}\)-invariant, it weight \(2\) with respect to the torus \(\cc^*_t\). We require
the differentials in a curved complex from \(\MF_n\) to have weight \(1\) with respect to \(\cc^*_t\) and it has
weight \(0\) with respect to \(\cc^*_a\). To simplify notations we do  not use  any extra indices to indicate such \(T_{sc}\)-equivariance.
We also use notation
\[\mathbf{q}^k\mathbf{t}^l\cdot \calF\]
for the matrix factorization \(\calF\) with the \(k\)-twisted \(\CC_a^*\)-action and \(l\)-twisted \(\cc_t^*\)-action.

Since the space \(\cx\) has \(B^2\)-action we can also twist a matrix factorization \(\calF\) by a representation of this group.
Given a characters \(\chi_l\) and \(\chi_r\) of the left and right factor in \(B^2\), the twisted matrix factorization is
denoted by
\[\calF\langle \chi_l,\chi_r\rangle.\]

To define convolution product in category \(\MF_n\) we introduce the convolution space \(\cx_{con}\) which is a fibered product:
\[\cx_{con}:=\sqrt{\cx}\ti_\frg\sqrt{\cx}\ti_\frg\sqrt{\cx}=\frg\ti (G\ti n)^2.\]

There are three \(G\ti B^3\)-equivariant maps \(\pi_{12},\pi_{23},\pi_{13}\) and the convolution product is defined by the predictable
formula:
\[\calf\star\calg:=\pi_{13*}(\CE_{\frn^{(2)}}(\pi_{12}^*(\calf)\ot\pi_{23}^*(\calg))^{T^{(2)}}).\]

Since the projections \(\pi_{ij}\) are smooth we can apply base change formula and the standard argument, that could be found in
\cite{ChrisGinzburg}. One derives with the use of the base change that thus defined product is associative.

\subsection{Knorrer reduction}
\label{sec:knorrer-reduction}
We can apply the Knorrer periodicity discussed in  section~\ref{sec:knorrer-periodicity} to  reduce the size of our working space \(\cx\).
Indeed, the pair of space and potential:
\[\bar{\cx}=\frb\ti G\ti \frn,\quad \Wr(X,g,Y)=\Tr(X\Ad_g(Y))\]
is \(B^2\)-equivariant with respect to the action:
\[(b_1,b_2)\cdot (X,g,Y)=(\Ad_{b_1}X,b_1gb_2,\Ad_{b_2}^{-1}Y).\]
Thus we can define the category of weakly \(U^2\)-equivariant and strongly \(T^2\)-equivariant matrix factorizations:
\[\overline{\MF}_n:=\MF_{B^2}(\bar{\cx},\Wr).\]
To illustrate some of our methods we provide a proof for the equivalence in 
\begin{prop}\label{prop:Kn}
  There is an equivalence of categories:
  \[\Psi:\overline{\MF}_n\to \MF_n.\]
\end{prop}

\begin{proof}
  First we observe that  the group \(G\) acts freely on the space \(\cx\) hence  we can take quotient by this group. The quotient can be
  implemented with help of the map:
  \[\cx\xrightarrow{q}\cx^\circ:=\frg\ti \frn\ti G\ti \frn,\quad q(X,g_1,Y_1,g_2,Y_2)=(\Ad_{g_1}^{-1}X,Y_1,g_1^{-1}g_2,Y_2).\]
  The potential \(W^\circ(X,Y_1,g,Y_2)=\Tr(X(Y_1-\Ad_gY_2))\) is the pull-back \(W_0=q^*(W)\) and the pull-back provides an
  equivalence \(q^*: \MF_n\simeq \MF_{B^2}(\cx^0,W_0)\).

  To complete our proof we fix notations for the truncation of a square matrix \(X\):
  \[X=X_++X_{--},\quad X_+\in \frn,\quad X_{--}^t\in \frb.\]
  The potential \(W^\circ\) can be written as a sum of \(\Wr\) and a quadratic term and thus we can apply the Knorrer periodicity:
  \begin{multline*}
    \Tr(X(Y_1-\Ad_gY_2))=    \Tr((X_++X_{--})(Y_1-\Ad_gY_2))=-\Tr(X_+(\Ad_gY_2))+\\
    \Tr(X_{--}(Y_1-\Ad_gY_2))=-\Tr(X_+(\Ad_gY_2))+
    \Tr(X_{--}(Y_1-\Ad_gY_2)_+).
  \end{multline*}
  The entries of the matrices \(X_{--},Y_1-(\Ad_gY_2)_+\) are the coordinates in the direction transversal to the
  subspace \(\bar{\cx}\) with coordinates \(X_+,g,Y_2\) and the Knorrer periodicity allows us to remove the quadratic
  term in the last formula.
\end{proof}

It is explained in \cite{OblomkovRozansky16} that the category \(\overline{\MF}_n\) has a monoidal structure \(\bar{\star}\) such that
that the functor \(\Psi\) sends it to the monoidal structure \(\star\).

\subsection{Braid groups}
\label{sec:braid-groups-1}

The affine braid group $\Br_n^{aff}$
is the group of braids whose strands may also wrap around a `flag pole'.
The group is generated by the standard generators $\sigma_i$, $i=1,\dots,n-1$
and  a braid $\Delta_n$ that wraps the last stand of the braid around the flag pole:

\begin{equation*}\label{AffineBraidGenerators}
\sigma_i =
\beginpicture
\setcoordinatesystem units <.5cm,.5cm>         
\setplotarea x from -5 to 3.5, y from -2 to 2    
\put{${}^{i+1}$} at 0 1.2      %
\put{${}^{i}$} at 1 1.2      %
\put{$\bullet$} at -3 .75      %
\put{$\bullet$} at -2 .75      %
\put{$\bullet$} at -1 .75      %
\put{$\bullet$} at  0 .75      
\put{$\bullet$} at  1 .75      %
\put{$\bullet$} at  2 .75      %
\put{$\bullet$} at  3 .75      %
\put{$\bullet$} at -3 -.75          %
\put{$\bullet$} at -2 -.75          %
\put{$\bullet$} at -1 -.75          %
\put{$\bullet$} at  0 -.75          
\put{$\bullet$} at  1 -.75          %
\put{$\bullet$} at  2 -.75          %
\put{$\bullet$} at  3 -.75          %
\plot -4.5 1.25 -4.5 -1.25 /
\plot -4.25 1.25 -4.25 -1.25 /
\ellipticalarc axes ratio 1:1 360 degrees from -4.5 1.25 center
at -4.375 1.25
\put{$*$} at -4.375 1.25
\ellipticalarc axes ratio 1:1 180 degrees from -4.5 -1.25 center
at -4.375 -1.25
\plot -3 .75  -3 -.75 /
\plot -2 .75  -2 -.75 /
\plot -1 .75  -1 -.75 /
\plot  2 .75   2 -.75 /
\plot  3 .75   3 -.75 /
\setquadratic
\plot  0 -.75  .05 -.45  .4 -0.1 /
\plot  .6 0.1  .95 0.45  1 .75 /
\plot 0 .75  .05 .45  .5 0  .95 -0.45  1 -.75 /
\endpicture
\qquad\hbox{and}\qquad
\Delta_n =
~~\beginpicture
\setcoordinatesystem units <.5cm,.5cm>         
\setplotarea x from -5 to 3.5, y from -2 to 2    
\put{$\bullet$} at -3 0.75      %
\put{$\bullet$} at -2 0.75      %
\put{$\bullet$} at -1 0.75      %
\put{$\bullet$} at  0 0.75      
\put{$\bullet$} at  1 0.75      %
\put{$\bullet$} at  2 0.75      %
\put{$\bullet$} at  3 0.75      %
\put{$\bullet$} at -3 -0.75          %
\put{$\bullet$} at -2 -0.75          %
\put{$\bullet$} at -1 -0.75          %
\put{$\bullet$} at  0 -0.75          
\put{$\bullet$} at  1 -0.75          %
\put{$\bullet$} at  2 -0.75          %
\put{$\bullet$} at  3 -0.75          %
\plot -4.5 1.25 -4.5 -0.13 /
\plot -4.5 -0.37   -4.5 -1.25 /
\plot -4.25 1.25 -4.25  -0.13 /
\plot -4.25 -0.37 -4.25 -1.25 /
\ellipticalarc axes ratio 1:1 360 degrees from -4.5 1.25 center
at -4.375 1.25
\put{$*$} at -4.375 1.25
\ellipticalarc axes ratio 1:1 180 degrees from -4.5 -1.25 center
at -4.375 -1.25
\plot -2 0.75  -2 -0.75 /
\plot -1 0.75  -1 -0.75 /
\plot  0 0.75   0 -0.75 /
\plot  1 0.75   1 -0.75 /
\plot  2 0.75   2 -0.75 /
\plot  3 0.75   3 -0.75 /
\setlinear
\plot -3.3 0.25  -4.1 0.25 /
\ellipticalarc axes ratio 2:1 180 degrees from -4.65 0.25  center
at -4.65 0
\plot -4.65 -0.25  -3.3 -0.25 /
\setquadratic
\plot  -3.3 0.25  -3.05 .45  -3 0.75 /
\plot  -3.3 -0.25  -3.05 -0.45  -3 -0.75 /
\endpicture.
\end{equation*}
The defining relations for this generators are
\begin{gather*}
   \sigma_{n-1}\cdot \Delta_n\cdot \sigma_{n-1}\cdot \Delta_n=  \Delta_n\cdot \sigma_{n-1}\cdot \Delta_n\cdot\sigma_{n-1},\\
   \sigma_i\cdot \Delta_n=\Delta_n\cdot \sigma_i,\quad i<n-1,\\
   \sigma_i\cdot \sigma_{i+1}\cdot \sigma_i=\sigma_{i+1}\cdot \sigma_i\cdot \sigma_{i+1},\quad i=1,\dots, n-2,\\
   \sigma_i\cdot \sigma_j=\sigma_j\cdot \sigma_i,\quad |i-j|>1.
\end{gather*}

The mutually commuting Bernstein-Lusztig (BL) elements $\Delta_i\in \Br_n^{aff}$ are defined as follows:
\begin{equation*}\label{BraidMurphy}
\Delta_i=\sigma_{i}\cdots \sigma_{n-2}\sigma_{n-1}
\Delta_n\sigma_{n-1}\sigma_{n-2}\cdots \sigma_{i} =
~~\beginpicture
\setcoordinatesystem units <.5cm,.5cm>         
\setplotarea x from -5 to 3.5, y from -2 to 2    
\put{${}^i$} at 1 1.2
\put{$\bullet$} at -3 0.75      %
\put{$\bullet$} at -2 0.75      %
\put{$\bullet$} at -1 0.75      %
\put{$\bullet$} at  0 0.75      
\put{$\bullet$} at  1 0.75      %
\put{$\bullet$} at  2 0.75      %
\put{$\bullet$} at  3 0.75      %
\put{$\bullet$} at -3 -0.75          %
\put{$\bullet$} at -2 -0.75          %
\put{$\bullet$} at -1 -0.75          %
\put{$\bullet$} at  0 -0.75          
\put{$\bullet$} at  1 -0.75          %
\put{$\bullet$} at  2 -0.75          %
\put{$\bullet$} at  3 -0.75          %
\plot -4.5 1.25 -4.5 -0.13 /
\plot -4.5 -0.37   -4.5 -1.25 /
\plot -4.25 1.25 -4.25  -0.13 /
\plot -4.25 -0.37 -4.25 -1.25 /
\ellipticalarc axes ratio 1:1 360 degrees from -4.5 1.25 center
at -4.375 1.25
\put{$*$} at -4.375 1.25
\ellipticalarc axes ratio 1:1 180 degrees from -4.5 -1.25 center
at -4.375 -1.25
\plot -3 0.75  -3 -0.1 /
\plot -2 0.75  -2 -0.1 /
\plot -1 0.75  -1 -0.1 /
\plot  0 0.75   0 -0.1 /
\plot -3 -.35  -3 -0.75 /
\plot -2 -.35   -2 -0.75 /
\plot -1 -.35   -1 -0.75 /
\plot  0 -.35    0 -0.75 /
\plot  2 0.75   2 -0.75 /
\plot  3 0.75   3 -0.75 /
\setlinear
\plot -3.2 0.25  -4.1 0.25 /
\plot -2.8 0.25  -2.2 0.25 /
\plot -1.8 0.25  -1.2 0.25 /
\plot -.8 0.25  -.2 0.25 /
\plot  .2 0.25  .5 0.25 /
\plot -3.3 -.25  .5 -.25 /
\ellipticalarc axes ratio 2:1 180 degrees from -4.65 0.25  center
at -4.65 0
\plot -4.65 -0.25  -3.3 -0.25 /
\setquadratic
\plot  .5 0.25  .9 .45  1 0.75 /
\plot  .5  -0.25  .9 -0.45 1 -0.75 /
\endpicture.
\end{equation*}

The  finite braid group \(\Br_n\) is the subgroup of the affine braid group with the generators \(\sigma_i\), \(i=1,\dots, n-1\).
Other words, we do not allow the braids to go around the pole.

There is a natural homomorphism $\mathfrak{fgt}\colon \Br_n^{aff}\to \Br_n$, geometrically it is defined by removing the flag pole. In particular we have:
$$ \mathfrak{fgt}(\Delta_n)=1,\quad \mathfrak{fgt}(\Delta_i)=\delta_i, \quad i=1,\dots,n-1.$$
The inclusion discussed above \(i_{fin}:\Br_n\to \Br_n^{aff}\) is a section of the flag forgetting map: \(\mathfrak{frg}\circ i_{fin}=1\).

\subsection{Braid action}
\label{sec:braid-action}
In this section we outline a construction of the homomorphisms from the (affine) braid group to our convolution algebras
of matrix factorizations. For  a geometric counter-part of the map \(\frg\) we need to introduce { \it stable } versions of
our categories of matrix factorizations.

Let us define the stable locus \(\bar{\cx}^{st,\bullet}\subset \bar{\cx}\ti V\) to be set of quadruples \((X,g,Y,v)\) that satisfy a open
condition:
\begin{equation} \label{eq:stab}
\cc\langle (\Ad_g^{-1}X)_+,Y\rangle v=V.\end{equation}

There is a natural projection \(\pi_V:\bar{\cx}\ti V\to \bar{\cx}\) and there is an open embedding map \(j_{st}: \bar{\cx}^{st}\to \bar{\cx}\)
where \(\bar{\cx}^{st}=\pi_V(\bar{\cx}^{\bullet,st})\). This map induces the pull-back map:
\[j_{st}^*:\overline{\MF}_n\to \overline{\MF}_n^{st}:=\MF_{B^2}(\bar{\cx}^{st},\Wr).
\]
It is shown in \cite[Lemma 13.3]{OblomkovRozansky16} that the  category \(\overline{\MF}_n\) has a natural structure of convolution
algebra.
The  main results of the papers \cite{OblomkovRozansky16}, \cite{OblomkovRozansky17} play the crucial role in the
construction of the knot invariant in  the next section.

\begin{theorem}
  There are homomorphisms of algebras:
  \[\Phi: \Br_n\to (\overline{\MF}_n^{st},\bar{\star}),\quad \Phi^{aff}:\Br_n^{aff}\to (\overline{\MF}_n,\bar{\star}).\]
  Moreover, the pull-back \(j_{st}^*\) is the homomorphism of the convolution algebras
  and
  \[j_{st}^*\circ \Phi^{aff}=\Phi\circ\fgt.\]
\end{theorem}

The fact that the pull-back morphism is an algebra homomorphism relies on the following shrinking lemma, for a proof
see lemma~12.3 in \cite{OblomkovRozansky16}. 

 \begin{lemma}\label{lem:shrink} Suppose $X$ is a quasi-affine variety and $\calF=(M,D)\in \MFs(X,W)$, $W\in \CC[X]$. The elements of \(\CC[X]\) act on
   \(\MF(X,W)\) by multiplication. Let us assume that
   the elements of the ideal $I=(f_1,\dots,f_m)\subset \CC[X]$ act by zero-homotopic endomorphisms on $\calF$ and $Z'\subset X$ is the
   zero locus of $I$. Let $Z\subset X$ be a subvariety
   defined by $J=(g_1,\dots,g_n)$ such that $Z\cap Z'=\emptyset$. Then $\calF$ is homotopic to $\calF|_{X\setminus Z}$ as
   matrix factorization over $\CC[X]$.
 \end{lemma}

 Essentially the lemma says that we can shrink our ambient space to any open neighborhood of the critical locus of the
 potential and such operation does not change the corresponding category of matrix factorizations.

 Let us also remark that there is another construction  for the affine braid group action on the similar
 category of matrix factorizations in \cite{ArkhipovKanstrup15a} but the
 precise relation between our construction and result of this paper is known to the author.

\section{Knot invariants}
\label{sec:knot-invariants}

\subsection{Geometric trace operator}
\label{sec:geom-trace-oper}
Let \(\mathfrak{b}_n\), \(\mathfrak{n}_n\) be Lie algebras of the group of upper, respectively strictly-upper triangular
\(n\times n\) matrices.
The free nested Hilbert scheme $\FHilb_n^{free}$ is a $B\times \CC^*$-quotient of the sublocus
$\widetilde{\FHilb}_{n}^{free}\subset \frb_n\times\frn_n\times V_n$ of the cyclic triples
$$\widetilde{\FHilb}_{n}^{free}=\{(X,Y,v)|\CC\langle X,Y\rangle v=V_n\},$$
here \(V_n=\CC^n\).
The usual nested Hilbert scheme  $\FHilb_{n}$ is the subvariety of $\FHilb^{free}_{n}$, it is defined by the condition that 
 $X,Y$ commute.

\begin{remark}
  There is a bit of confusion  in the notations, what we denote here by \(\FHilb_n\)   is denoted in \cite{OblomkovRozansky16}
  by \(\Hilb_{1,n}\) and by \(\FHilb_n(\cc)\) in \cite{GorskyNegutRasmussen16}.
\end{remark}

The torus $T_{sc}=\CC^*\times\CC^*$ acts on $\FHilb_{n}^{free}$ by scaling the matrices. We denote by $D^{per}_{T_{sc}}(\FHilb_{n}^{free})$ a derived category of the two-periodic complexes of the $T_{sc}$-equivariant quasi-coherent sheaves on $\FHilb_{n}^{free}$. Let us also denote by $
\mathcal{B}^\vee$ the descent of the
trivial vector bundle $V_n$ on $\widetilde{\FHilb}^{free}_{n}$ to the quotient $\FHilb^{free}_{n}$. Respectively, \(\mathcal{B}\) stands for the dual of
\(\mathcal{B}^\vee\). Below we construct for every $\beta\in \Br_n$ an
element $$\mathbb{S}_\beta\in D^{per}_{T_{sc}}(\FHilb_{n}^{free})$$ such that space of hyper-cohomology of the complex:
$$\mathbb{H}^k(\mathbb{S}_\beta):=\mathbb{H}(\mathbb{S}_\beta\otimes \Lambda^k\mathcal{B}) $$
defines an isotopy invariant.

\begin{theorem}\cite{OblomkovRozansky16}\label{thm:mainOR} For any $\beta\in\Br_n$ the doubly graded space
$$ H^k(\beta):=\mathbb{H}^{(k+\textup{writh}(\beta)-n-1)/2}(\mathbb{S}_\beta)$$
is an isotopy invariant of the braid closure $L(\beta)$.
\end{theorem}

The variety $\widetilde{\FHilb}_{n}^{free}$ embeds inside $\calXr$ via a map $j_e:(X,Y,v)\rightarrow (X,e,Y,v)$. The diagonal copy $B=B_\Delta\hookrightarrow B^2$
respects the embedding $j_e$ and since $j_e^*(\Wr)=0$, we obtain a functor:
$$ j_e^*: \MF_{B_n^2}(\calXr^{st},\Wr)=\overline{\MF}_n^{st}\rightarrow \MF_{B_\Delta}(\widetilde{\FHilb}_{n}^{free},0).$$
Respectively, we get a geometric version of "closure of the braid" map:
$$\mathbb{L}: \MF_{B_n^2}(\calXr^{st},\Wr)=\overline{\MF}_n^{st}\to D^{per}_{T_{sc}}(\FHilb_{n}^{free}).$$
The main result of \cite{OblomkovRozansky16} could be restated in more  geometric term via geometric trace map:
$$ \mathcal{T}r: \Br_n\rightarrow D^{per}_{T_{sc}}(\FHilb_{n}^{free}), \quad \mathcal{T}r(\beta):=\oplus_k \mathbb{L}(\Phi(\beta)\otimes \Lambda^k\mathcal{B}).$$

The  above mentioned complex \(\bbS_\beta\)  is the complex \(\mathbb{L}(\Phi(\beta))\). The differentials in
the complex \(\bbS_\beta\) are of degree \(t\) thus the differentials are invariant with respect to the anti-diagonal
torus \(T_a\). Hence the forgetful functor \(\chi: D^{per}_{T_{sc}}(\FHilb^{free})\to D^{per}_{T_a}(\FHilb^{free})\) could
be composed with \(K\)-theory functor \(\mathrm{K}: D^{per}_{T_{a}}(\FHilb^{free})\to K_{T_{a}}(\FHilb^{free})\).
The  composite functor \(\mathrm{K}\circ \chi\) is closely related to decategorification and the classical Ocneanu-Jones
trace, we  discuss the Ocneaunu-Jones trace  \(\Tr^{OJ}\) and related theorem of Markov in the next subsection.

\newcommand{\trvsp}{\mathrm{Vect-gr}}
\newcommand{\trgrh}{\mathrm{H}^{(3)}}
\begin{theorem}\cite{OblomkovRozansky16}\label{thm:HOMFLY} The composition $\mathbb{H}\circ \mathcal{T}r: \Br_n\rightarrow D^{per}_{T_{sc}}(pt)$ categorifies the Jones-Oceanu trace:
  \[\Tr^{OJ}(\beta)=\dim_{a,q}\mathrm{K}\circ\chi\circ \mathbb{H}\circ \mathcal{T}r(\beta),\]
  where the \(q\)-grading comes \(T_a\)-action and \(a\)-grading is from the exterior powers of \(\cb\)
\end{theorem}

\subsection{OJ trace and Markov theorem}
\label{sec:oj-trace-markov}

As we discussed before every link \(L\) in \(\mathbb{R}^3\) is isotopic to the closure of a braid \(L=\rmL(\beta)\), \(\beta\in \Br_n\).
On the other hand it is clear that such a presentation is not unique. Markov theorem describes the non-uniqueness explicitly and
thus provides an algebraic description of the set \(\mathfrak{L}\) of the isotopy equivalence classes of the links.

\begin{theorem} The closure operation \(\rmL\) identifies
  the set \(\mathfrak{L}\) of isotopy class of links in \(S^3\) and the set of equivalence classes:
  \[\bigcup_n \Br_n/\sim\]
  where the equivalence relation is generated by the elementary  equivalences:
  \begin{gather}
    \label{eq:conj} \alpha\cdot \beta\sim \beta\cdot\alpha,\quad \alpha,\beta\in \Br_n\\
    \label{eq:mark} \Br_{n+1}\ni\alpha\cdot \sigma^{\pm 1}\sim \alpha,\quad \alpha\in \Br_n.
  \end{gather}
\end{theorem}

If we have homomorphism \(\Tr\) from the braid group to some field \(F\)
that respects the relations \eqref{eq:mark} then the value \(\Tr(\beta)\in F\) is  an isotopy invariant of the
closure \(\mathrm{L}(\beta)\). In practice it is hard to classify such homomorphisms, however the great discovery
of Ocneanu and Jones is that one can classify such homomorphisms if we pass to a quotient \(H_n\) of the braid group.

The Hecke algebra \(H_n\) is generated by \(g_i\), \(i=1,\dots,n-1\) modulo relations:
\[g_ig_{i+1}g_{i}=g_{i+1}g_ig_{i+1},
  \quad i=1,\dots,n-2,\]
\[g_i-g_i^{-1}=q-q^{-1},\quad i=1,\dots, n-1.\]

There is a natural algebra homomorphism $\pi:\Br_n\to H_n$, \(\sigma_i\mapsto g_i\). It is shown in  \cite{Jones87} that
there is a unique homomorphism \(\Tr^{OJ}:\bigcup_n H_n\to \QQ(a,q)\) that satisfies relations \eqref{eq:mark} and normalizing relation
\[\Tr^{OJ}(\mathbf{1})=\frac{a-a^{-1}}{q-q^{-1}}.\]
The corresponding invariant is \(\Tr^{OJ}(\beta)\in\QQ(a,q)\) is also known as HOMFLYPT invariant, \(\mathrm{HOMFLYPT}(\beta)\).

Thus
 the formula \eqref{eq:HOMFLY} from the introduction and theorem~\ref{thm:HOMFLY} state that there is a specialization of the
graded dimension of 
\(\HHH(\beta)\)
that becomes a HOMFLYPT invariant. Let us recall that the  space \(\HHH(\beta)\) up to some elementary grading shift
is equal:
\[\mathbb{H}^*(\Phi(\beta)\ot \Lambda^\bullet\calB)^B.\]

This space naturally has four gradings: \(*,\bullet\) and \(T_{sc}\)-grading. However, only three of these gradings are
invariant with respect to the Markov moves: \(*\) is not preserved by the  moves. The first grading is \(\bullet\), we call it \(a\)-grading, since it is
responsible for the \(a\)-variable in the HOMFLYPT polynomial specializations.
The  other two gradings come from \(T_{sc}\)-action:
\[\deg( X_{ij})=q^2,\quad \deg(Y_{ij})=q^{-2}t^{-2}.\]

To specialize to the HOMFLYPT polynomial we need to set \(t=-1\) or more geometrically, we need to restrict torus
\(T_{sc}\)-action on the space \(\HHH(\beta)\) to the action of the anti-diagonal torus,
we denote the specialized category by \(\overline{\MF}_n^{st,a}\). To be more precise, the category \(\MF_n^{st,a}\)
is a category of matrix factorizations on \(\overline{\cx}^{st}\) with the potential \(\Wr\) which are
\(B^2_n\)-weakly equivariant, \(G\ti T^2\ti T_a\)-strongly equivariant.

As we mentioned before this torus is special because under this specialization the differentials in the curved
complexes from \(\overline{\MF}_n^{st}\) become torus invariant, hence there is a well-defined functor:
\[\mathrm{K}: \quad \overline{\MF}_n^{st,a}\to \mathrm{K}_{T_a}(\MF_n^{st,a}).\]
This K-theory functor turns the homotopy equivalence~\eqref{eq:trian} into the relation:
\[[\calC_+]=q^{-1}([\calC_\parallel]-[\calC_\bl\langle -\chi_1,-\chi_1\rangle]).\]
Thus the combination of  the relations \eqref{eq:blb-sqr} and \eqref{eq:left1}, \eqref{eq:right1} imply  the quadratic relation in the Hecke algebra.

The  Markov move relations \eqref{eq:mark} hold for the invariant \(\HHH(\beta)\) and 
in the section~\ref{sec:markov-relations}  we  discuss the main idea of our proof of the Markov moves for \(\HHH\). We need some
details of the braid group action construction for the Markov move argument, therefore we outline the construction in the next section.



\section{Geometric realization of the affine braid group}
\label{sec:geom-real-affine}

\subsection{Induction functors}
\label{sec:induction-functors}
The standard parabolic subgroup \(P_k\) has Lie algebra generated by \(\frb\) and \(E_{i+1,i}\), \(i\ne k\).
Let us define space  \(\calXr(P_k):=\frb\times P_k \times \frn\) and let us also use notation
\(\calXr(\GL_n)\) for \(\calXr(\GL_n)\). There is a natural embedding \(\bar{i}_k:\calXr(P_k)\rightarrow\calXr\) and
a natural projection \(\bar{p}_k:\calXr(P_k)\rightarrow\calXr(\GL_k)\times\calXr(\GL_{n-k})\). The space
\(\calXr(\GL_k)\ti\calXr(\GL_{n-k})\) is equipped with  a \(B_k^2\ti B_{n-k}^2\)-invariant
potential \(\Wr^{(1)}+\Wr^{(2)}\) which is a sum of pull-backs of the potentials \(W\) along the projection
on the first and the second factors. Moreover, we have:
\begin{equation}\label{eq:WW}
  \bar{i}_k^*(\Wr)=\bar{p}_k^*(\Wr^{(1)}+\Wr^{(2)}).
\end{equation}

Since the embedding \(\bar{i}_k\) satisfies
the conditions for existence of the push-forward and the relation  \eqref{eq:WW} between the potentials holds, we can define the induction functor:
\[\overline{\ind}_k:=\bar{i}_{k*}\circ \bar{p}_k^*: \MF_{B_k^2}(\calXr(\GL_k),\Wr)\times\MF_{B_{n-k}^2}(\calXr(\GL_{n-k}),\Wr)\rightarrow\MF_{B_n^2}(\calXr(\GL_{n}),\Wr)\]

Similarly we define space  \(\calXr_{fr}(P_k)\subset\frb\times P_k \times \frn\times V\) as an open subset defined by the stability condition \eqref{eq:stab}.
The last space has a natural projection map
\[\bar{p}_k:\calXr_{fr}(P_k)\rightarrow\calXr(\GL_k)\times\calXr_{fr}(\GL_{n-k})\]
and
the embedding \(\bar{i}_k:\calXr_{fr}(P_k)\rightarrow\calXr_{fr}(\GL_{n})\) and we can define the induction functor:
\[\overline{\ind}_k:=\bar{i}_{k*}\circ \bar{p}_k^*: \MF_{B_k^2}(\calXr(\GL_k),\Wr)\times\MF_{B_{n-k}^2}(\calXr_{fr}(\GL_{n-k}),\Wr)\rightarrow\MF_{B_n^2}(\calXr_{fr}(\GL_{n}),\Wr)\]

It is shown in section 6 (proposition 6.2) of \cite{OblomkovRozansky16} that the functor \(\overline{\ind_k}\) is the homomorphism of the convolution algebras:
\[\overline{\ind}_k(\mathcal{F}_1\boxtimes\mathcal{F}_2)\bar{\star} \overline{\ind}_k(\mathcal{G}_1\boxtimes\mathcal{G}_2)=
  \overline{\ind}_k(\mathcal{F}_1\bar{\star}\mathcal{G}_2\boxtimes\mathcal{F}_2\bar{\star}\mathcal{G}_2).\]
To define the non-reduced version of the induction functors one needs to introduce the space \(\calX^\circ(\GL_n)=\frg\times \GL_n\times\frn\times\frn\) which is the
slice to \(\GL_n\)-action on the space \(\calX\). In particular, the potential \(W\) on this slice becomes:
\[W(X,g,Y_1,Y_2)=\Tr(X(Y_1-\Ad_g(Y_2))).\]
Similarly to the case of the reduced space, one can define the  space \(\calX^\circ(P_k):=\frg\times P_k\times\frn\times\frn\)  to be the locus of and the corresponding
maps \(i_k:\calX^\circ(P_k)\rightarrow\calX^\circ(\GL_n)\), \(p_k:\calX^\circ(P_k)\rightarrow\calX^\circ(\GL_k)\times\calX^\circ(\GL_{n-k})\). Thus we get a version of the induction functor for non-reduced spaces:
\[{\ind}_k:=i_{k*}\circ p_k^*: \MF_{B_k^2}(\calX(\GL_k),W)\times\MF_{B_{n-k}^2}(\calX(\GL_{n-k}),W)\rightarrow\MF_{B_n^2}(\calX(\GL_{n}),W)\]

It is shown in proposition 6.1 of \cite{OblomkovRozansky16} that the Kn\"orrer functor is compatible with the induction functor:
\[\ind_k\circ(\Phi_k\times\Phi_{n-k})=\Phi_n\circ\ind_k.\]

\subsection{Generators of the finite braid group action}
\label{sec:gener-braid-group}
Let us define \(B^2\)-equivariant embedding \(i: \calXr(B_n)\rightarrow\calXr\), \(\calXr(B):=\frb\times B\times \frn\).
The pull-back of \(\Wr\) along the map \(i\) vanishes and the embedding \(i\) satisfies the conditions for existence of the push-forward
\(i_*:\MF_{B^2}(\calXr(B_n),0)\rightarrow \MF_{B^2}(\calXr(\GL_n),\Wr)\). We denote by \(\underline{\CC[\calXr(B_n)]}\in\MF_{B^2}(\calXr(B_n),0)\) the
matrix factorization with zero differential that is concentrated  only in even homological degree. As it is shown in proposition 7.1 of \cite{OblomkovRozansky16} the
push-forward
\[\bar{\mathds{1}}_n:=i_{*}(\underline{\CC[\calXr(B_n)]})\]
is the unit in the convolution algebra. Similarly, \(\mathds{1}_n:=\Phi(\bar{\mathds{1}}_n)\) is also a unit in non-reduced case.

Let us first discuss the case of the braids on two strands. The key to construction of the braid group action in \cite{OblomkovRozansky16} is the following factorization in the case
\(n=2\):
$$\Wr(X,g,Y)=y_{12}(2g_{11}x_{11}+g_{21}x_{12})g_{21}/\det,$$
where \(\det=\det(g)\) and
$$ g=\begin{bmatrix} g_{11}&g_{12}\\ g_{21}& g_{22}\end{bmatrix},\quad X=\begin{bmatrix} x_{11}&x_{12}\\ 0& x_{22}\end{bmatrix},\quad Y=\begin{bmatrix} 0& y_{12}\\0&0\end{bmatrix}$$
Thus we can define the following strongly equivariant Koszul matrix factorization:
\[\bar{\mathcal{C}}_+:=(\CC[\calXr]\otimes \Lambda\langle\theta\rangle,D)\in\MF^{str}_{B^2}(\calXr,\Wr),\]
\[  \quad D=\frac{g_{12}y_{12}}{\det}\theta+\bigg(g_{11}(x_{11}-x_{22})+g_{21}x_{12}\bigg)\frac{\partial}{\partial\theta},\]
where \(\Lambda\langle\theta\rangle\)   is the exterior algebra with one generator.

This matrix factorization corresponds to the positive elementary braid on two strands.
Using the induction functor we can extend the previous definition to the case of the arbitrary number of strands. For that we introduce an insertion
functor:
\[\overline{\Ind}_{k,k+1}:\MF_{B_2^2}(\calXr(\GL_2),\Wr)\rightarrow\MF_{B_n^2}(\calXr(\GL_n),\Wr)\]
\[\overline{\Ind}_{k,k+1}(\calF):=\overline{\ind}_{k+1}(\overline{\ind}_{k-1}(\bar{\mathds{1}}_{k-1}\times \calF)\times\bar{\mathds{1}}_{n-k-1}),\]
and similarly we define non-reduced insertion functor \[\Ind_{k,k+1}:\MF_{B_2^2}(\calX(G_2),W)\rightarrow\MF_{B_n^2}(\calX(G_n),W).\]
Thus we define the generators of the braid group as follows:
\[\bar{\calC}_+^{(k)}:=\overline{\Ind}_{k,k+1}(\overline{\ind}_{k-1}( \bar{\calC}_+)),\quad \bar{\calC}_+^{(k)}:=\Ind_{k,k+1}(\ind_{k-1}( \calC_+)).\]

The section 11 of \cite{OblomkovRozansky16} is devoted to the proof of the braid relations between these elements:
\[\bar{\calC}^{(k+1)}_+\bar{\star}\bar{\calC}^{(k)}_+\bar{\star}\bar{\calC}^{(k+1)}_+=\bar{\calC}^{(k)}_+\bar{\star}\bar{\calC}_+^{(k+1)}
  \bar{\star}\bar{\calC}_+^{(k)},\]
\[\calC^{(k+1)}_+\star\calC^{(k)}_+\star\calC^{(k+1)}_+=\calC^{(k)}_+\star\calC_+^{(k+1)}
  \star\calC_+^{(k)}.\]

Let us now discuss the inversion of the elementary braid. In view of inductive definition of the braid group action, it is sufficient
to understand the inversion in the case \(n=2\).

Thus we define:
\[\calC_-:=\calC_+\langle-\chi_1,\chi_2\rangle\in\MF_{B^2}(\calX(\GL_2),W),\]
where \(\chi_1\), \(\chi_2\) are the standard generators of the character group of \(\CC^*\ti \CC^*=T\subset B_2 \).  The definition of \(\bar{\calC}_-\) is similar. It could be shown that the definition of \(\calC_-\) is actually symmetric with respect to the left-right twisting: \[\calC_-= \calC_+\langle\chi_2,-\chi_1\rangle.\]


\begin{theorem}
  We have:
  \begin{equation}
    \label{eq:1}
    \calC_+\star\calC_-=\mathds{1}_2.
  \end{equation}
\end{theorem}

\section{Sample computation}
\label{sec:sample-computation}
In this section  we would like to show an example of the convolution algebra computations. But before I would
expand a little bit our discussion of the basic matrix factorizations in the case of \(n=2\).

\subsection{Basic matrix factorizations of rank $2$}
\label{sec:basic-matr-fact}

We have shown in the previous section that the potential \(\Wr\) is a product of three factors and we
used this fact to define the matrix factorization \(\bar{\calC}_+\). However, it is clear that there are
two more natural  matrix factorizations for this potential:
\[\bar{\mathcal{C}}_\parallel:=(\CC[\calXr]\otimes \Lambda\langle\theta\rangle,D_\parallel,0,0),\quad \bar{\mathcal{C}}_\bullet:=(\CC[\calXr]\otimes \Lambda\langle\theta\rangle,D_\bullet,0,0)\in\MF_{B^2}(\calXr,\Wr),\]
\[  D_\parallel=\frac{g_{21}}{\det}\theta+y_{12}\tilde{x}_0\frac{\partial}{\partial\theta},\quad D_\bullet= \frac{g_{21}}{\det}\tilde{x}_0\theta+y_{12}\frac{\partial}{\partial\theta},\quad
\tilde{x}_0=g_{11}(x_{11}-x_{22})+g_{21}x_{12}.\]

One of the matrix factorizations is actually a cone of the morphism between the other two:
\begin{equation}\label{eq:trian}
  [\bar{\calC}_\parallel\xrightarrow{\phi} \bar{\calC}_\bullet\langle -\chi_1,-\chi_1\rangle]\sim \mathbf{q}\mathbf{t}\cdot \bar{\calC}_+
  \end{equation}
with map \(\phi\) defined by
\[
  \begin{tikzcd}
    \mathbf{t}^{-1}\cdot R\langle \chi_2,-\chi_1\rangle\arrow[d, "g_{21}", shift left]\arrow[r,"1"]&
    R\langle \chi_2,-\chi_1\rangle\arrow[d,"g_{21}\tilde{x}_0",shift left]\\
    R\arrow[u,"y_{12}\tilde{x}_0",shift left]\arrow[r,"\tilde{x}_0"]&R\langle 0,-\chi_1\rangle\arrow[u,"y_{12}",shift left],
\end{tikzcd}
\]
where \(R=\CC[\overline{\cx}]\).
This relation  is crucial for our discussion of the connection with the Oceanus-Jones traces

\subsection{Details on the convolution product}
\label{sec:deta-conv-prod}

The convolution product inside the category \(\MF_{B^2}(\calXr,\Wr)\) is a bit tricky to define and we refer reader to
our paper \cite{OblomkovRozansky16} where the convolution product is constructed and used
for the computations for \(n=3\). On the other  hand the space \(\calX\) is bigger than the space \(\calXr\) but the construction of the
convolution is more straightforward. The space \(\calX^\circ:=\calX/\GL_n=\frg\ti\frn\ti \GL_n\ti\frn\) is intermediate between these
two spaces and we  choose to work with this slightly bigger space to
make our exposition simpler.

The  space \(\calX^\circ\) and the relevant potential \(W^\circ\) appeared already in the proof of the proposition~\ref{prop:Kn}.
Let us spell out the definition of the convolution structure for the elements \(\calF,\calG\in \MF_{B^2}(\calX^\circ,W^\circ)\):
\[\calF\star\calG:=\pi_{13*}^\circ(\CE(\pi^{\circ*}_{12}(\calF)\ot\pi_{23}^{\circ*}(\calG))),\]
where we used the convolution space \(\cx_{cnv}^\circ:=\frg\ti\frn\ti\GL_n\ti\frn\ti\GL_n\ti\frn\) and the \(B^3\)-equivariant maps are
\[\pi_{12}^\circ(X,Y_1,g_{12},Y_2,g_{23},Y_3)=(X,Y_1,g_{12},Y_2),\quad \pi_{23}^\circ(X,Y_1,g_{12},Y_2,g_{23},Y_3)=(\Ad_{g_{12}}X,Y_2,g_{23},Y_3),\]
\[ \pi_{13}^\circ(X,Y_1,g_{12},Y_2,g_{23},Y_3)=(X,Y_1,g_{12}g^{-1}_{23},Y_3).\]

To write the versions \(\calC_\parallel,\calC_\bullet,\calC_+\) of the matrix factorizations from above we need more precise notations
for the Koszul matrix factorizations. We use the matrix notation
\[\begin{bmatrix} a_1&b_1&\theta_1\\
    \vdots&\vdots&\vdots\\
    a_m&b_m&\theta_m
   \end{bmatrix}.
\]
for the matrix factorization from \(\MF(X,F)\) with the differential \(D=\sum_{i=1}^ma_i\theta_i+b_i\frac{\partial}{\partial \theta_i}\) acting on
\(\cc[X]\ot\Lambda^\bullet[\theta]\).

Let us also fix coordinates on the space \(\calX^\circ=\frg\ti \frb\ti G\ti\frb\):
\[X=\begin{bmatrix}x_0+\tr/2& x_1\\ x_{-1}&-x_0+\tr/2\end{bmatrix},\quad Y_i=\begin{bmatrix}0&y_i\\0&0\end{bmatrix}\quad g=
\begin{bmatrix}
a_{11}&a_{12}\\  a_{21}&a_{22}
\end{bmatrix},
\]
where \(\tr=\tr X\). We also denote by \(\delta_1,\delta_2\) the generators of \(\Lie(U^2)\), \(U^2\subset B^2\).
We also only indicate non-trivial actions of \(\delta_i\), that is if no action of \(\delta_i\) is given then
this action is trivial.

\def\crXbl{\calC_\bullet}
\def\crXpr{\calC_\parallel}
\def\crXp{\calC_+}

\def\exa{ a }
\def\exav#1{ \exa_{#1} }
\def\exaoo{ \exav{11} }
\def\exaot{ \exav{12} }
\def\exato{ \exav{21} }
\def\exatt{ \exav{22} }

\def\xv#1{ x_{#1} }
\def\xo{ \xv{1} }
\def\xz{ \xv{0} }
\def\xmo{ \xv{-1}}

\def\yv#1{ y_{#1} }
\def\yo{ \yv{1} }
\def\yt{ \yv{2} }
\def\yi{ \yv{i} }

\def\yh{y_3}

\def\exa{ a }
\def\exav#1{ \exa_{#1} }
\def\exaoo{ \exav{11} }
\def\exaot{ \exav{12} }
\def\exato{ \exav{21} }
\def\exatt{ \exav{22} }

\def\exav#1{a_{#1}}
\def\exaoo{\exav{11}}
\def\exaot{\exav{12}}
\def\exato{\exav{21}}
\def\exatt{\exav{22}}
\def\exbv#1{b_{#1}}
\def\exboo{\exbv{11}}
\def\exbot{\exbv{12}}
\def\exbto{\exbv{21}}
\def\exbtt{\exbv{22}}
\def\excv#1{c_{#1}}
\def\excoo{\excv{11}}
\def\excot{\excv{12}}
\def\excto{\excv{21}}
\def\exctt{\excv{22}}

\def\xXtl{ \tilde{\xX} }
\def\dgBs{ \deg_{\mflg^2}}
\def\dltv#1{ \delta_{#1}}
\def\dlto{ \dltv{1} }
\def\dltt{ \dltv{2}}
\def\dlth{ \dltv{3}}

\def\txz{ \tilde{x}_0 }
\def\ttxz{ \tilde{\tilde{x}}_0 }

With this conventions we have
the matrix factorization of the identity braid has the form
\begin{equation*}
\crXpr =
\begin{bmatrix}
\xmo & \yo - \yt\exaoo^2 & \theta_1
\\
\yt\txz
& \exato & \theta_2
\end{bmatrix},\quad
\dlto\theta_1 = -2\yt\exaoo\theta_2.
\end{equation*}
The blob matrix factorization has the form
\begin{equation*}
\crXbl=
\begin{bmatrix}
\xmo & \yo - \yt\exaoo^2 & \theta_1
\\
\exato \txz
& \yt & \theta_2'
\end{bmatrix},\quad
\dlto\theta_1 = -2\exato\exaoo\theta_2'
\end{equation*}
or equivalently
\begin{equation*}
\crXbl=\begin{bmatrix}
\xmo & \yo & \theta'_1
\\
-\exaoo^2\xmo+
\exato \txz
& \yt & \theta_2'
\end{bmatrix},\quad \theta'_1=\theta_1 +\exatt^2\theta_2',\quad
\dlto\theta'_1 = 0
\end{equation*}

The matrix factorization of the positive intersection is 
\begin{equation}
\crXp=
\begin{bmatrix}
\xmo & \yo - \yt\exaoo^2 & \theta_1
\\
\txz & \exato\yt & \theta_2
\end{bmatrix}
,\quad
\dlto\theta_1 = -2\exaoo\theta_2.
\end{equation}

\subsection{Computation}
\label{sec:computation}

Now we are ready to do our sample computation.

\begin{prop} In the convolution algebra of \(\MF_{\GL_n\ti B^2}(\cx,W)\) we have:
  $$\crXbl\langle 0,\chi_1\rangle\star\crXbl\langle 0,\chi_1\rangle=\crXbl\langle\chi_1,\chi_1\rangle
  \oplus\crXbl\langle\chi_2,\chi_1\rangle.$$
\end{prop}

\begin{proof}
  
  Let us fix some more notation for the coordinates on the spaces that appear in our constructions. For the  group
  elements elements in the product \(\calX_{conv}^\circ=\frg\ti\frn\ti\GL_2\ti\frn\ti\GL_2\ti\frn\) we use notations \(a,b\) and
  for the non-zero elements of  upper-triangular matrices in the product  we use  \(y_1,y_2,y_3\). We also add prime 
  to the conjugate of \(X\): \(X'=\Ad_aX\).
  
  Thus the   matrix factorization \(\calC''=\pi_{12}^{\circ *}(\calC_\bullet)\ot\pi_{23}^{\circ *}(\calC_\bullet)\) is the following Koszul
  matrix factorization:
\begin{align*}
\calC''& =
\begin{bmatrix}
\xmo & \yo - \yt\exaoo^2 & \theta_1
\\
\exato(2\xz \exaoo + \xo\exato) & \yt & \theta_2
\\
\xmo' & \yt - \yh\exboo^2 & \theta_3
\\
\exbto(2\xz' \exboo + \xo'\exbto) & \yh & \theta_4
\end{bmatrix},&
\dlto\theta_1 &= -2\exaoo\theta_2,&
\dltt\theta_3 &=-2\exboo\theta_4.
\end{align*}
By making suitable linear change of \(\theta_1\mapsto \theta_1+2a_{11}\theta_2,\theta_2\mapsto\theta_2\)
and \(\theta_3\mapsto \theta_3+b_{11}\theta_4,\theta_4\mapsto\theta_4\)  we can make the first
simplification of this matrix factorization:
\begin{align*}
\calC''& =
\begin{bmatrix}
\xmo & \yo  & \theta_1
\\
-\exaoo^2\xmo+\exato(2\xz \exaoo + \xo\exato) & \yt & \theta_2
\\
\xmo' & \yt &\theta_3
\\
-\exboo^2\xmo'+\exbto(2\xz' \exboo + \xo'\exbto) & \yh & \theta_4
\end{bmatrix},& \delta_i\theta_j=0.
\end{align*}
We use the third row  to remove $\yt$ from the other rows:
\begin{align*}
\calC''& =
\begin{bmatrix}
\xmo & \yo  & \theta_1
\\
-\exaoo^2\xmo+\exato(2\xz \exaoo + \xo\exato) & 0 & \theta_2'
\\
0 & \yt &\theta_3
\\
-\exboo^2\xmo'+\exbto(2\xz' \exboo + \xo'\exbto) & \yh & \theta_4
\end{bmatrix},& \theta_2' &= \theta_2 - \theta_3.
\end{align*}
Since \(\theta_3\) is \(B^2\) invariant element, we can now remove the third row altogether and work over the
ring \(R'=\cc[\cx^{\circ}_{conv}]/(y_2)\).

We can also use  the relation
\[
-\exboo^2\xmo'+\exbto(2\xz' \exboo + \xo'\exbto) =
-\excoo^2\xmo + \excto(2\xz\excoo+\xo\excto)
\]
to arrive to 
\begin{align*}
\crXbl \star \crXbl& =
\begin{bmatrix}
\xmo & \yo  & \theta_1
\\
-\exaoo^2\xmo+\exato(2\xz \exaoo + \xo\exato) & 0 & \theta_2'
\\
-\excoo^2\xmo + \excto(2\xz\excoo+\xo\excto)
& \yh & \theta_4
\end{bmatrix}
\end{align*}

Doing couple more simple row transformations, that change the basis in the space \(\langle \theta_1,\theta'_2,\theta'_4\rangle\),
we arrive to a simplified presentation of \(\crXbl \star \crXbl\):

\begin{align*}
\calC''& =
\begin{bmatrix}
\xmo & \yo -c_{11}^2y_3 & \theta_1'
\\
\exato(2\xz \exaoo + \xo\exato) & 0 & \theta_2''
\\
 \excto(2\xz\excoo+\xo\excto)
& \yh & \theta_4'
\end{bmatrix}
\end{align*}

Now let us notice that   the top and the bottom lines of the last Koszul complex are $\delta_2$-invariant and together
they form a Koszul matrix  factorization \(\pi_{13}^{\circ,*}(\calC_\bl)\).
On the other hand  the middle line has only one non-trivial differential and
to complete the proof we  need to compute the \ChE homology
\[H^*_{\Lie}(\frn,R''\xrightarrow{f} R'')^T,\quad f=-\exato(2\xz \exaoo + \xo\exato), \]
where \(R'=R''\ot \CC[\GL_2]\) with last copy of \(\GL_2\) has coordinates \(c_{ij}\).

The  space \(\Spec(R'')\) has coordinates \(a,y_1,y_2,x\) and the Lie algebra \(\frn\) only acts on the entries of
the matrix \(a\):
\[\delta_2 a_{i2}=-a_{i1},\quad \delta_2 a_{i1}=0.\]
The differential in complex for \(H^*_{\Lie}\) is exactly \(\delta_2\) hence
\[H^0(\frn,R'')=\cc[y_1,y_2,x,a_{11},a_{21},\mathrm{det}^{\pm 1}],\quad H^1(\frn,R'')=\cc[y_1,y_2,x,a,\mathrm{det}^{\pm 1}]/(a_{11},a_{21}),\]
where \(\det=\det a\).  Now we can extract the torus invariant part:
\[(H^*(\frn,R'')\otimes \chi_1)^{T_{sc}}=(H^0(\frn,R'')\otimes \chi_1)^{T_{sc}}=\langle a_{11},a_{21}\rangle.\]

Finally, let us observe that the function \(f\) is quadratic on \(a\) hence its induced action action on \((H^*(\frn,R'')\otimes \chi_1)^{T_{sc}}\) is trivial and 
 the statement follows.

\end{proof}

Now let us  derive the formula~\eqref{eq:blb-sqr} from the above computation. For  that let us recall the stable locus
\(\bar{\cx}^{st}\) is union of two open subsets: \(U_y=\{y\ne 0\}\), \(U_x=\{(\Ad_g^{-1}X)_{12}\ne 0\}\).
On the open set \(U_y\) the matrix factorization \(\bar{\calC}_\bullet\) contracts since \(y\) is one of the differentials
of the curved complexes. Thus we can safely restrict our attention to the open locus \(U_x\) but on this locus
\((\Ad_g^{-1}X)_{12}\ne 0\). Since  the weights of \(T_{sc}\ti B^2\) on this non-vanishing elements are:
\[\mathrm{weight}( (\Ad_g^{-1}X)_{12})=\mathbf{q}^2\cdot\langle 0,-\chi_1+\chi_2\rangle,\quad
\mathrm{weight}(\det(a))=\langle\chi_1+\chi_2,\chi_1+\chi_2\rangle\]
we can trade the Borel action weight for \(\mathbf{q}\) \(\chi_1\)-shifts for \(\mathbf{q}\)-shifts:
\begin{equation}\label{eq:left1}
  \bar{\calC}_\bl\langle \chi'+\chi_1,\chi''\rangle=\mathbf{q}^{2}\bar{\calC}_\bl\langle \chi'-\chi_2,\chi''-2\chi_2\rangle,
\end{equation}
\begin{equation}\label{eq:right1}
  \bar{\calC}_\bl\langle \chi',\chi'+\chi_1\rangle=\mathbf{q}^{-2}\bar{\calC}_\bl\langle \chi',\chi''+\chi_2\rangle.
\end{equation}
Finally, we refer to the theorem~\ref{thm:kill-JM} that implies that the pull-back \(j_{st}^*\) turns the shifts
\(\langle \chi_2,0\rangle\) and \(\langle 0,\chi_2\rangle\) to the trivial  \(B^2\)-equivariant shift.

\section{Markov relations}
\label{sec:markov-relations}

The  first Markov relation is equivalent to \(\HHH\) being a trace, that is we need to show that the functor  \(\HHH\) is
constant on the conjugacy classes inside \(\Br_n\). In fact one can show stronger statement. Before we state the
this stronger statement let us discuss the connection with usual flag Hilbert schemes.

\subsection{Sheaves on the flag Hilbert scheme}
\label{sec:sheaves-flag-hilbert}

The usual flag Hilbert scheme $\FHilb_{n}$
is a subvariety of $\Hilb_{1,n}^{free}$ defined by the commutativity constraint on $X,Y$:
\[ [X,Y]=0.\]
It turns out that the support of the homology of the complex $\bbS_\beta$ is contained in $\Hilb_{1,n}^L$. Hence the sheaf homology of the complex is the  sheaf
\[\calS_\beta=\calS_\beta^{odd}\oplus\calS_\beta^{even}:=\mathcal{H}^*(\FHilb_{n}^{free},\mathbb{S}_\beta)\] on $\Hilb_{1,n}$ and we immediately have the following:
\begin{theorem}\label{thm: spec seq}
There is a spectral sequence with $E_2$ term being
$$ (\textup{H}^*(\FHilb_{n},\calS_\beta\otimes\Lambda^k\calB),d)$$
$$
d: \textup{H}^k(\FHilb_{n},\calS^{odd/even}_\beta\otimes\Lambda^k\calB)\to
\textup{H}^{k-1}(\FHilb_{1,n},\calS^{even/odd}_\beta\otimes\Lambda^k\calB),$$
that converges to $\mathbb{H}^k(\beta)$.
\end{theorem}

The theorem follows almost immediately  from the main theorem \ref{thm:mainOR} and 
the proposition~\ref{prop:crit}. Moreover the sheaf \(\calS_\beta\) is actually is a conjugacy invariant:

\begin{theorem}\cite{OblomkovRozansky16}
  For any \(\alpha,\beta\in \Br_n\) we have:
  \[\calS_{\alpha\cdot\beta}\simeq \calS_{\beta\cdot\alpha}.\]
\end{theorem}

The  argument could be found in the cited paper, here we illustrate the idea by showing that
\begin{equation}\label{eq:m1}
  \calS_{\sigma_i\sigma_j\sigma_k}\simeq \calS_{\sigma_j\sigma_k\sigma_i}.\end{equation}
Indeed, let us introduce the space \(\cx_3\subset \gl_n\times (\GL_n\times \frn_n)^3\) defined by the
constraint requiring the cyclic product of the group elements to be one.
There is  a natural \(B^3\)-action
and \(B^3\)-equivariant projections :
\[pr_i:\cx_3\to \gl_n\times \frn_n,\quad pr_i(X,g_1,Y_1,g_2,Y_2,g_3,Y_3)=(X,Y_i).\]
Respectively,  we also have projections \(\pi^\circ_{12},\pi_{23}^\circ,\pi_{31}^\circ:\cx_3\to\cx^\circ\) and
\[\bbS_{\sigma_i\sigma_j\sigma_k}=pr_{1*}(\calC),\quad \bbS_{\sigma_j\sigma_k\sigma_i}=pr_{2*}(\calC),\quad
  \calC=\CE_{\frn^3}(\pi_{12}^{\circ *}(\calC^{(i)}_+)\otimes \pi_{23}^{\circ*}(\calC^{(j)}_+)\otimes \pi_{31}^{\circ*}(\calC^{(k)}_+))\]

On the critical locus of \(\pi^{\circ*}_{i,i+1}(W^\circ)\) we have \(Y_i=\Ad_{g_i}Y_{i+1}\) hence on the critical locus the conjugation by \(g_1\) intertwines the projections \(pr_1\) and \(pr_2\) the isomorphism \eqref{eq:m1} follows.

In the argument above we ignore the stability conditions but can check that the shrinking lemma~\ref{lem:shrink} implies
that the argument above works even after we impose the stability conditions.

\subsection{Second Markov move}
\label{sec:second-markov-move}

The  second Markov relation is more subtle and a proof of this relation is arguably the most valuable result of \cite{OblomkovRozansky16}.
To convey the main idea of the proof we explain why it holds for the braids on two strands. In this case we need to compare
the homology of the closure of \(\sigma_1^{\pm 1}\) with the homology of unknot, so let us first do the most trivial case of
the braids on one strands since \(\mathrm{L}(\mathbf{1}_1)\), \(\mathbf{1}_1\in \Br_1\) is manifestly the unknot.

Indeed, for \(n=1\) we have \(\bar{\cx}_1=\cc\times\cc^*\times 0\) and \(j_e\) embeds \(\widetilde{\FHilb}_1^{free}=\cc\times 1\times 0\)
inside \(\bar{\cx}_1\). The group \(B_1=\cc^*\) acts trivially on \(\widetilde{\FHilb}_1^{free}\) and thus \( \FHilb_1=\cc\) and
\(\mathbb{S}_e=j^*_e(\calO_{\bar{\cx}_1})=\calO_{\cc}\) and \(\calB_1\) is the trivial bundle. We conclude then:
\[\dim_{q,t}H^0(\mathbf{1}_1)=\dim_{q,t}H^1(\mathbf{1}_1)=\frac{1}{1-q^2}. \]

Now let us explore the geometry of 
the free Hilbert scheme \(\FHilb_2^{free}\). Let us fix coordinates on the space \(\widetilde{\FHilb}_2^{free}\subset \frb\times \frn\times V\):
\[X=\begin{bmatrix}x_{11}& x_{12}\\ 0& x_{22}\end{bmatrix},\quad Y=\begin{bmatrix} 0&y \\0&0,
  \end{bmatrix},
  \quad v=\begin{bmatrix} v_1\\v_2
    \end{bmatrix}.
\]

Since we have the stability condition \(\cc\langle X,Y\rangle v=\cc^2\) and both \(X,Y\) are upper-triangular, we must have
\(v_2\ne 0\). Thus after conjugating by the appropriate upper-triangular matrix we could assume that \(v_2=1,v_1=0\), let us denote
this vector by \(v^0\). It is also elementary to see that
\[\cc\langle X,Y\rangle v^0=\cc^2 \mbox{ if and only if } x_{12}y\ne 0.\]
Also the stabilizer of \(v^0\) is \(\cc^*\) that scales \(x_{12},y\) and preserves \(x_{11},x_{22}\). Hence we have shown:
\[\FHilb_2^{free}=\mathbb{P}^1\times \CC^2,\]
the projection \(p\) on \(\CC^2\) is given by the coordinates \(x_{11},x_{22}\).

Let us contrast the geometry of \(\FHilb^{free}_2\) with the geometry of \(\FHilb_2\). The  discussion
in this paragraph is not used in the proof below and is just an illustration of
difficulties of the geometry of the flag Hilbert scheme. The  condition \([X,Y]=0\) is equivalent to
the constraint:
\[y(x_{11}-x_{22})=0.\]
Hence the fibers of the projection \(p:\FHilb_2\to \CC^2\)  are points outside of the diagonal \(x_{11}=x_{22}\) and the fibers
are projective lines \(\mathbb{P}^1\) over the diagonal.

Next let us recall that the matrix factorization for the simple positive crossing  is \(\calC_+=|\tilde{x},yg_{21}|\).
Since \(\tilde{x}|_{g=1}=(x_{11}-x_{22})\), the pull-back \(j_{e}(\calC_+)\) is the Koszul complex that is homotopy equivalent to the
structure sheaf of \(\mathbb{P}^1\times \CC\). Finally, the tautological vector bundle is a sum of the line bundles
\(\calB^\vee=\calO\oplus \calO(-1)\), hence:
\[H^0(\sigma_1)=H^*(\calO_{\mathbb{P}^1\times \cc})=\CC[x_{11}],\quad
  H^1(\sigma_1)=H^*(\calB^\vee)=H^*(\calO_{\mathbb{P}^1\times \cc}\oplus \calO_{\mathbb{P}^1\times \cc}(-1) )=\CC[x_{11}],\]
\[H^2(\sigma)=H^*(\det(\calB))=H^*(\calO_{\mathbb{P}^1\times \cc}(-1))=0.\]
By our construction the matrix factorization for the negative crossing
is differs by a line bundle twist  from the one for positive crossing. In particular, we have
\(j_e(\calC_-)=\calO_{\mathbb{P}^1\times \cc}(-1)\) and can compute the homology:
\[H^0(\sigma_1^{-1})=H^*(\calO_{\mathbb{P}^1\times \cc}(-1))=0,\quad
  H^2(\sigma^{-1})=H^*(\det(\calB)\otimes \calO(-1))=H^*(\calO_{\mathbb{P}^1\times \cc}(-2))=\CC[x_{11}], \]
  \[H^1(\sigma_1^{-1})=H^*(\calB^\vee\otimes \calO(-1))=H^*(\calO_{\mathbb{P}^1\times \cc}(-1)\oplus \calO_{\mathbb{P}^1\times \cc}(-2) )=\CC[x_{11}],
   \]

Thus we have shown \(H^k(\sigma)=H^k(\mathbf{1}_1)\) and \(H^{k+1}(\sigma^{-1})=H^k(\mathbf{1}_1)\) as we expected.


Respectively, we can use nested nature of the scheme \(\FHilb_n\) to define the intermediate map:
\[\pi: \FHilb^{free}_n\to \cc\times \FHilb^{free}_{n-1},\]
where the first component of the map \(\pi\) is \(x_{11}\) and the second component is just forgetting of the first
rows and rows of the matrices \(X,Y\) and the first component of the vector \(v\). Let us also fix notation for the line bundles on
 \(\FHilb_n^{free}\): we denote by \(\calO_k(-1)\) the line bundle induced from the twisted trivial bundle \(\calO\otimes\chi_k\).
It is quite elementary to
show

\begin{prop}
  The fibers of the map \(\pi\) are projective spaces \(\mathbb{P}^{n-1}\) and
  \begin{enumerate}
  \item \(\calB_n/\pi^*(\calB_{n-1})=\calO_n(-1).\)
  \item \(\calO_n(-1)|_{\pi^{-1}(z)}=\calO_{\mathbb{P}^{n-1}}(-1).\)  
  \end{enumerate}
\end{prop}

We can combine the last proposition with the observation that the total homology \(H^*(\mathbb{P}^{n-1},\calO(-l))\) vanish
if \(l\in (1,n-1)\) and is one-dimensional for \(l=0,n\):

\begin{cor}
  For any \(n\) we have:
  \begin{itemize}
  \item \(\pi_*(\Lambda^k\calB_n)=\Lambda^k\calB_{n-1}\)
  \item \(\pi_*(\calO_n(m)\ot \Lambda^k\calB_n)=0\) if \(m\in [-n+2,-1]\).
    \item \(\pi_*(\calO_n(-n+1)\ot \Lambda^k\calB_n)=\Lambda^{k-1}\calB_{n-1}[n]\)
  \end{itemize}
\end{cor}

The geometric version of the Markov move is the following
\begin{theorem} For any \(\beta\in \Br_{n-1}\) we have
  \[\mathbb{H}^k(\beta\cdot\sigma_1)=\mathbb{H}^k(\beta),\quad \mathbb{H}^k(\beta\cdot\sigma_1)=\mathbb{H}^{k-1}(\beta).\]
\end{theorem}
\begin{proof}[Sketch of a proof]
  The main technical component of the proof is the careful analysis of the matrix factorizations \(\bar{\calC}_{\beta\cdot \sigma^{\pm 1}}\MF(\bar{\cx}_n,\Wr)\).
   It is shown in \cite{OblomkovRozansky16} that this curved complex \(\bar{\calC}_{\beta\cdot\sigma_1^{\epsilon}}\) has form:
   \begin{equation}\label{dia:big}
     \begin{tikzcd}
              \calC'\ar[r,bend right]\ar[rrr,bend right]\ar[rrrrr,bend right]&
       \calC'\ot V\ar[l,dotted]\ar[r,bend right]\ar[rrr,bend right]&
       \calC'\ot\Lambda^2V\ar[l,dotted]\ar[r,bend right]\ar[rrr,bend right]&
       \calC'\ot\Lambda^3V\ar[l,dotted]\ar[r,bend right]&
      \calC'\ot\Lambda^2V\ar[l,dotted]\ar[r,bend right]&\cdots\ar[l,dotted]
  \end{tikzcd}
\end{equation}
  where \(\calC'=p_1^*(\bar{\calC}_{\beta})\), \(V=\CC^{n-2}\),
 the dotted arrows are the differentials of the Koszul complex for the ideal \(I=(g_{13},\dots,g_{1n})\) where \(g_{ij}\) are the coordinates on the
  group inside the product \(\bar{\cx}_n=\frb_n\ti \GL_n\ti \frn_n\). Thus after the pull-back \(j_e^*\) the dotted arrows of the curved complex
  vanish and we only left with the arrows going from the left to right.

  Now we would like to compute \(\pi_*(j^*_e(\bar{\calC}_{\beta\cdot\sigma_1^{\pm 1}})\ot \Lambda^k\calB_n)\) and  here we can apply the previous corollary.
  Thus if \(\epsilon=1\) then only the left extreme term of \(j_e^*\) of the complex \eqref{dia:big} survive the push-forward \(\pi_*\). Since
  the non-trivial arrows of \(j_e^*\) of \eqref{dia:big} all  are the solid arrows which are  going the left to the right, the contraction of the \(\pi_*\)-acyclic terms
  do not lead to appearance of new correction arrows thus conclude that
  \[\pi_*(j^*_e(\bar{\calC}_{\beta\cdot\sigma_1})\ot \Lambda^k\calB_n)=j_e^*(\bar{\calC}_\beta\ot \Lambda^k\calB_{n-1}).\]

  If \(\epsilon=-1\) then only the right extreme term of \(j_e^*\) of the complex \eqref{dia:big} survive the push-forward \(\pi_*\). Hence the
  similar argument as before implies:
    \[\pi_*(j^*_e(\bar{\calC}_{\beta\cdot\sigma_1^{-1}})\ot \Lambda^k\calB_n)=j_e^*(\bar{\calC}_\beta\ot \Lambda^{k-1}\calB_{n-1}).\]
\end{proof}
\section{Chern functor and localization}
\label{sec:chern-funct-local}
The   theorem~\ref{thm: spec seq} provides a theoretical method for constructing a sheaf on the flag Hilbert scheme that
contains all the information about the knot homology of the closure of the braid \(\rmL(\beta)\).
However, it is hard to use this method for actually computing knot homology.

The first complication comes from the fact that the space \(\FHilb_{n}\) is very singular and working with this
space requires extra level of care and technicalities \cite{GorskyNegutRasmussen16}. We will explain how
one can circumvent this complication with the Chern functor from the next subsection.

The second complication
comes from possible non-vanishing of the differential \(d\) in the theorem, one would like to
avoid the spectral sequence that do not degenerate at the second step. The differential vanishes automatically
if for example \(\calS_\beta^{odd}\) vanishes, this kind of property is probably related
to the {\it parity } property in \cite{EliasHogancamp17}
for Soergel bimodel model of the knot homology. Again the Chern functor helps with finding braids that have the
parity property, as we explain in the end of the section.

\subsection{Chern functor}
\label{sec:chern-functor}

In the paper \cite{OblomkovRozansky18a} we construct a pair of functors which we call a Chern functor and a co-Chern functor:
\begin{equation}
\label{eq:mnchcchf}
  \begin{tikzcd}
    \MF^{\st}_n\arrow[rr,bend left,"\CH^{st}_{loc}"]&& D^{per}_{T_{sc}}(\Hilb_n)\arrow[ll,bend left,"\HC^{st}_{loc}"]
  \end{tikzcd},
\end{equation}
where \(\Hilb_n\) is the Hilbert scheme of \(n\) points on \(\CC^2\), while \(D^{per}_{T_{sc}}(\Hilb_n)\) is the derived category of
two-periodic \(T_{sc}\)-equivariant complexes on the Hilbert scheme. In the same paper we prove
\begin{theorem}\cite{OblomkovRozansky18a}
For every \(n\)  we have
\begin{itemize}
\item The functors \(\CH^{\rst}_{\rloc}\) and \(\HC^{\rst}_{\rloc}\) are adjoint.
\item The functor \(\HC^{\rst}_{\rloc}\) is monoidal.
\item The image of \(\HC^{\rst}_{\rloc}\) commutes with the elements \(\Phi(\beta)\), \(\beta\in \Br_n\).
\end{itemize}

\end{theorem}

As a manifestation of the categorified Riemann-Roch formula, we obtain a new interpretation for the triply-graded homology:
\begin{theorem}\cite{OblomkovRozansky18a}\label{thm:main}
  For any \(\beta\in \Br_n\) we have:
  \[\mathrm{HHH}(\beta)=\Hom(\calO,\CH^{\rst}_{\rloc}(\Phi(\beta))\otimes \Lambda^\bullet\cb).\]
\end{theorem}

Let us outline the construction of the Chern functor in the next subsection.

\subsection{The Chern \(\CH\)}
\label{sec:chern-funct-deta}
First we will construct the functor between the categories \(\MF\) and \(\MF_{\dr}\) where the last category is defined as
a stable version of the category of equivariant matrix factorizations:
\[
\MF_{\dr}:=\MF_{G}(\scc,W_{\dr}),\quad
\scc=\frg\times G\times \frg,\quad W_{\dr}(Z,g,X)=\Tr(X(Z-\Ad_gZ)),
\]
the group $\GL_n$ acting on components of $\scc$ by conjugation. The stable version of
the category is defined as category of matrix factorizations on  the slightly enlarged space:
\[\scc^{st}\subset \scc\ti V,\quad (Z,g,X,v)\in \scc^{st}\mbox{ iff } gv=v.\]
Both stable and unstable versions of the categories fit into the diagram:
\begin{equation*}
  \begin{tikzcd}
    \MF^{\bullet}\arrow[rr,bend left,"\CH^{\bullet}"]&& \MF_{\dr}^{\bullet}\arrow[ll,bend left,"\HC^{\bullet}",pos=0.435]
  \end{tikzcd},
\end{equation*}
where \(\bullet\) can be either \(st\) or \(\emptyset\).

To lighten the exposition we explain only the construction for the functors \(\CH\) and \(\HC\), the stable
version is an easy modification of the construction, see \cite{OblomkovRozansky18a}.
 We need two auxiliary spaces in order to define the Chern functor: 
\[\scz^0_{\CH}=\frg\times G\times\frg\times G\times\frn,\quad \scZ_{\CH}=\frg\times G\times\frg\times G\times\frb\]
The action of \(G\times B\) on these spaces is
\[(k,b)\cdot (Z,g,X,h,Y)=
  (\Ad_{k}(Z),\Ad_k(g),\Ad_k(X),khb,\Ad_{b^{-1}}(Y))\]
and the invariant potential is
\[W_{\CH}(Z,g,X,h,Y)=\Tr(X(\Ad_{gh}(Y)-\Ad_{h}(Y))).\]
The spaces \(\scC\) and \(\scX\) are endowed with the standard \(G\times B^2\)-equivariant structure, the action of  \(B^2\) on \(\scC\) is trivial.
The following maps
\[\pi_{\Dr}:\mathscr{Z}_{\CH}\rightarrow \mathscr{C},\quad f_\Delta:\scZ_{\CH}^0\rightarrow \mathscr{X}\quad j^0:\scZ^0_{\CH}\to \scZ_{\CH}.\]
\[\pi_{\Dr}(Z,g,X,h,Y)=(Z,g,X),\quad f_\Delta(Z,g,X,h,Y)=(X,gh,Y,h,Y)\]
are  fully equivariant if we restrict the \(B^2\)-equivariant structure on
\(\scX\) to the \(B\)-equivariant structure via the diagonal embedding
\(\Delta:B\rightarrow B^2\). Note that $j^0$ is the inclusion map.

The kernel of the Fourier-Mukai transform is the Koszul matrix factorization
\[\mathrm{K}_{\CH}:=[X-\Ad_{g^{-1}}X,\Ad_hY-Z]\in \MF(\scZ_{\CH},\pi_{\Dr}^*(W_{\Dr})-f^*_\Delta(W)).\]
and we define the Chern functor:
\begin{equation}\label{eq:CH}
  \CH(\calC):=\pi_{\dr*}(\CE_{\frn}(\mathrm{K}_{\CH}\otimes (j^0_*\circ f^*_\Delta(\calC)))^{T}).
\end{equation}
We also define the co-Chern functor $\HC$ as the adjoint functor that goes in the opposite direction:
\(\HC\colon\MF_{\Dr}\rightarrow \MF.\)
Thus, the functor $\HC$ is the composition of adjoints of all the functors that appear in the formula~\eqref{eq:CH}.



The product \(\scZ_{\CH}\times B\) has a \(B\times B\)-equivariant structure: for $(p,g)\in \scZ_{\CH}\times B $ we define
\[
(h_1,h_2)\cdot (p, g) = (h_1\cdot p,h_1 g h_2^{-1})
\]
%
Then the following map is \(B^2\)-equivariant:
\[\tilde{f}_{\Delta}:\scZ_{\CH}^0\times B\rightarrow \scX\times B,\]
\begin{equation*}\tilde{f}_\Delta(Z,g,X,h,Y,b)=(X,gh,Y,hb,\Ad_bY,b).\end{equation*}
The map \(\tilde{f}_\Delta\) is a composition of the projection along the first factor of \(\scZ_{\CH}\) and
the embedding inside \(\scX\times B\). The embedding is defined by the formula
\[\Ad_bY_1=Y_2,\]
so it is a regular embedding. Thus since
\[j^{0*}(\mathrm{K}_{\CH}\otimes \tilde{\pi}_{\Dr}^*(\calD))\in\MF_{G\times B^2}(\scZ_{\CH}\times B,\tilde{f}_\Delta^*(W)),\]
where \(\tilde{\pi}_{\Dr}:\scZ_{\HC}\times B\rightarrow\scC\) is a natural extension of map
\(\pi_{\Dr}\) by the projection along \(B\),
we have a well-defined matrix factorization \(\tilde{f}_{\Delta *}\circ j^{0*}(\mathrm{K}_{\CH}\otimes \pi_{\Dr}^*(\calD))\in \MF_{G\times B^2}(\scX\times B,\pi_B^*(W)),\) where \(\pi_B\) is the projection along the last factor. Now we can define:
\begin{equation}\label{eq:HC}\HC(\calD):=\pi_{B*}(\tilde{f}_{\Delta*}\circ j^{0*}(\mathrm{K}_{\CH}\otimes \pi_{\Dr}^*(\calD))).\end{equation}

\subsection{Linear Koszul duality}
\label{sec:line-kosz-dual}

We need to relate the category \(\MF_{\Dr}^{st}\) and the category \(D^{per}_{T_{sc}}(\Hilb)\). This relation is a particular example
of the linear Koszul duality. Let us discuss the linear Koszul duality in general.

Derived algebraic geometry is explained in many places, here we explain it  in the most elementary setting sufficient for our needs.
%

Initial data for an affine derived complete  intersection is a collection of elements
\(f_1,\dots,f_m\in\CC[X]\). It determines the differential graded algebra
\[\calR=(\CC[X]\otimes \Lambda^*U,D),\quad D=\sum_{i=1}^mf_i\frac{\partial}{\partial \theta_i},\]
where \(\theta_i\) from a basis of \(U=\CC^m\).

More generally, given a dg algebra \(\calR\) such that \(H^0(\calR)=\calO_Z\) we say that \(\mathop{Spec}(\calR)\) is a dg
scheme with underlying scheme \(Z\). Respectively, we define dg category of coherent sheaves on
\(\mathop{Spec}(\calR)\) as
\[\Coh(\mathop{Spec}(\calR))=\frac{\{\text{bounded complexes of finitely generated  $\calR$  dg modules}\}}{\{\text{quasi-isomorphisms}\}}.\]

Consider a potential on \(X\times U\):
\[W=\sum_{i=1}^m f_i(x)z_i,\]
where \(z_i\) is a basis of \(U^*\) dual to the basis \(\theta_i\). For the Koszul matrix factorization:
\[\MF(X\ti U,W)\ni\mathrm{B}=(\calR\otimes \CC[U], D_B),\quad D_B=\sum_{i=1}^m z_i\theta_i+f_i\frac{\partial}{\partial \theta_i}.\]
and for a \((M,D_M)\)  dg module over \(\calR\),  the tensor product
\[\mathrm{KSZ}_U(M):=M\otimes_{\CC[X]\otimes\Lambda^*(U)} \mathrm{B}\]
is an object of \(\MF(X\times U,W)\) with the differential \(D=D_M\otimes 1+ 1\otimes D_B\). The map \(\mathrm{KSZ}_U\) extends to a functor between triangulated categories:
\[\mathrm{KSZ}_U\colon \Coh(\Spec(\calR))\rightarrow \MF(X\times U,W).\]

The functor in the other direction is based on the dual matrix factorization:
\[\MF(X\times U,-W)\ni B^*=(\calR\otimes \CC[U],D_B^*),\quad D^*_B=-\sum_{i=1}^m z_i\theta_i+f_i\frac{\partial}{\partial \theta_i},\]
\[\mathrm{KSZ}_U^*\colon\MF(X\times U,W)\rightarrow \Coh(\Spec(\calR)),\quad \mathrm{KSZ}^*_U(\calF):=
  \Hom_{\calR}(\calF\otimes_{\CC[X\times U]}\mathrm{B}^*,\calR).\]

\begin{theorem} The compositions of the functors:
  \[\mathrm{KSZ}_U\circ \mathrm{KSZ}_U^*,\quad \mathrm{KSZ}_U^*\circ \mathrm{KSZ}_U\]
  are autoequivalences of the corresponding categories. 
\end{theorem}

Proof of this theorem could be found in \cite{Isik10}, \cite{ArkhipovKanstrup15} or one can consult \cite{OblomkovRozansky18a} for a more
streamlined argument.
\subsection{Linearization}
\label{sec:linear}

We would like to apply the  Koszul duality in our situation. The complication  in our case is that
we want to eliminate the group factor in the space \(\mathscr{C}^{st}\) but the group is not a linear space.
Thus we have to restrict ourselves to the neighborhood of the identity and linearize the potential in
this neighborhood and as we explain below it could be done with localization.

A coordinate substitution $Y = U g^{-1}$
on our main variety \(\scC\) makes the potential tri-linear:
\[\underline{W}_{\Dr}(X,U,g)=\Tr(X[U,g])=W_{\Dr}(X,Ug^{-1},g).\]

Thus we introduce linearized categories:
\[\underline{\MF}^\bullet_{\Dr}:=\MF_{G}(\underline{\scC}^{\bullet},\underline{W}_{\rlin}),\]
where \(\underline{\scC}^\bullet\) is obtained from \(\scC^\bullet\) by taking the closure of \(G\)
 inside \(\frg\).

Since \(j_{G}\colon\scC^\bullet\hookrightarrow\underline{\scC}^\bullet\) is an open embedding, the pull-back functor \(j_G^*\)
is a localization functor and we denote
\[\mathrm{loc}^\bullet\colon\MF^\bullet_{\Dr}\rightarrow\underline{\MF}^\bullet_{\Dr}\]
for this functor.

\begin{prop}\cite{OblomkovRozansky18a}\label{prop:loc-iso}
  The functors \(\mathrm{loc}^{\rst}\) are isomorphisms.
\end{prop}

Since the potential \(\underline{W}\) is linear as a function of \(g\in\frg\) and the scaling torus \(T_{sc}\) does not act on
\(g\), we obtain a pair of mutually inverse functors:
\[
  \begin{tikzcd}
    \underline{\MF}^\bullet_{\Dr}\arrow[rr,bend left,"\mathrm{KSZ}_{\frg}^*"]&&\Coh^\bullet\arrow[ll,bend left,"\mathrm{KSZ}_\frg"]
  \end{tikzcd}
\]
here \(\Coh^{st}\) is the two-periodic derived category \(D^{per}(\Hilb)\) and
\(\Coh\) is the DG category of the commuting variety.

The  functors that we wanted to construct are defined by the composing the functors:

\[\CH^{\rst}_{\rloc}:=\CH^{\rst}\circ (\mathrm{loc}^{\rst})^{-1}\circ \mathrm{KSZ}^*_\frg: \MF^{\bullet}\rightarrow \mathrm{D}^{\rper}(\Hilb_n(\CC^2)).\]

The localization functor does not seem to be invertible in  the case of  \(\bullet=\emptyset\), however a construction of the
functor in the opposite direction does not require invertibility of the localization:
\[\HC^\bullet_{\rloc}:=\HC^\bullet\circ \mathrm{loc}^\bullet\circ \mathrm{KSZ}_\frg\colon\Coh^\bullet\rightarrow\MF^\bullet.\]

\subsection{Localization formulas}
\label{sec:local-form}

The advantage of this new interpretation is that the Hilbert scheme is smooth, unlike the flag Hilbert scheme which
is a homological support of \(\mathcal{E}\mathrm{xt}(\Phi(\beta),\Phi(1))\). So the complexes
on \(\Hilb\) are  more manageable than their flag counter-part. In support of this expectation, we apply the Chern functor to the
Jusys-Murphy (JM) subgroup inside \(\Br_n\) together with the parity property and prove an explicit localization formula for the sufficiently positive elements of the JM subgroup.

Recall that the JM subgroup is generated by the elements
\[
\delta_i=\sigma_i\sigma_{i+1}\dots\sigma_{n-1}^2\dots\sigma_{i+1}\sigma_i.
\]
It is not hard to see that these elements mutually commute and that the full twist from the introduction is the product:
\[\mathrm{FT}=\prod_{i=1}^{n-1}\delta_i.\]

It is expected that \(\CH^{\rst}_{\rloc}\) applied to the matrix factorization corresponding to the
sufficiently positive element of JM algebra is a sheaf supported in one homological degree, we
state the precise conjecture below. Modulo this geometric conjecture we have
 a (conditional on the conjecture) formula for the corresponding homology of the links.

\begin{theoremc}\label{thm:loc1} For any \(n\) there are \(N,M>0\) such that
 for a vector $\vec{b}\in \ZZ^{n-1}$ with \(a_{i+1}-a_{i}>N, a_2>M\)
  the $(q,t,a)$-character of the homology of the closure of the braid $\prod_{i=2}\delta^{b_i}$
 is given by the formula
 \[
  \dim_{a,Q,T}\mathrm{HHH}(\prod_{i=2}^{n}\delta^{b_i})=
\sum_T \prod_i\frac{z_i^{b_i}(1+az_i^{-1})}{1-z^{-1}}\prod_{1\le i<j\le n}\zeta(\frac{z_i}{z_j}),\]
  where \(\zeta(x)=\frac{(1-x)(1-QTx)}{(1-Qx)(1-Tx)}\), \(Q=q^2,T=t^2/q^2\).  The last sum is over all standard Young tableaux with
  \(z_i=Q^{a'(i)}T^{l'(i)}\), \(a',l'\) are co-arm and co-leg of the square the standard tableau with the square with the label \(i\).
\end{theoremc}

The  proof  has two components. The first component is concerned with actual computation of the matrix factorization
\(\Phi(\prod_{i=2}\delta_i^{b_i})\). This  computation is an easy consequence of our construction of \(\Phi^{aff}\):

\begin{theorem}\cite{OblomkovRozansky17} \label{thm:kill-JM}For any $i=1,\dots, n$ we have
$$\Phi^{aff}(\Delta_i)=\Phi^{aff}(1)\langle\chi_i,0\rangle.$$
\end{theorem}

In particular, we show in \cite{OblomkovRozansky17} that the pull-back \(j_{st}^*\) sends \(\Phi^{aff}(1)\langle\chi_i,0\rangle\)
to the trivial line bundle. Since \(\delta_i=\fgt(\Delta_i)\) we conclude the following 

\begin{cor}\label{cor:pos}
  For any \(\beta\in\Br_n\)   and \(b_i,M\in \zz\) we have
  \[\Phi(\beta\cdot\prod_{i=2}^n\delta_i^{b_i})=\Phi(\beta)\langle \vec{b},0\rangle,\]
  \[\CH^{\rst}_{\rloc}(\Phi(\beta\cdot \mathrm{FT}^M))=\CH^{\rst}_{\rloc}(\Phi(\beta))\otimes \det(\calB)^{\otimes M},\]
  where \(\mathrm{FT}=\prod_{i=2}^n\delta_i\) is the full-twist braid.
\end{cor}

Thus we can apply the first formula from the corollary and get an explicit Koszul matrix factorization describing the
desired curved complex:
\[\Phi(\prod_{i=2}^n\delta_i^{b_i})=\calC_\parallel\langle \vec{b},0\rangle. \]

It is much harder though to compute Chern functor \(\CH(\calC_\parallel\langle \vec{b},0\rangle)\). It is
expected \cite{GorskyNegutRasmussen16} that \(\CH(\calC_\parallel)\) is a celebrated Procesi vector bundle and
finding an explicit description for this vector bundle is notoriously hard \cite{Haiman02}. So at the moment we do not
have an explicit statement for \(\CH(\calC_\parallel\langle \vec{b},0\rangle)\) but we believe the following weaker conjecture
could be proved by inductive argument from the work of Haiman.

\begin{conj}\cite{OblomkovRozansky18c} There is \(N>0\) such that for any \(\vec{a}\) such that \(a_{i+1}-a_i>N \) the two periodic
  complex \(\CH^{\rst}_{\rloc}(\calC_\parallel\langle\vec{a},0\rangle)\) is homotopy equivalent to the sheaf concentrated in even homological
  degree.
\end{conj}

On the other hand \(\det(\calB)\) is an ample line bundle on \(\Hilb_n(\CC^2)\) and hence the assumptions of the theorem~\ref{thm:loc1} and
corollary~\ref{cor:pos} imply that \(\CH^{\rst}_{\rloc}(\calC_\parallel\langle\vec{b})\rangle\)   is homotopy equivalent to the sheaf with no higher homology.
The differential in the complex \(\calC_\parallel\) has \(T_{sc}\) degree \(\mathrm{t}\), respectively by \(\mathrm{t}\)-twisting even component
of \(\calC_\parallel\) we obtain the curved complex \(\calC^{ev}_\parallel\) with \(T_{sc}\)-invariant differential.
From the discussion above we have:
\[H^*(\CH^{\rst}_{\rloc}(\calC_\parallel\langle\vec{a},0\rangle))=H^0(\CH^{\rst}_{\rloc}(\calC^{ev}_\parallel\langle\vec{a},0\rangle))
  =\chi(\CH^{\rst}_{\rloc}(\calC^{ev}_\parallel\langle\vec{a},0\rangle))=\chi(\mathbb{S}_1^{ev}\langle\vec{a},0\rangle),\]
where \(\mathbb{S}_1^{ev}\) 
is the version of \(\mathbb{S}_1\) with \(\mathrm{t}\)-twisted even component. There is well-defined image \(\mathrm{K}(\calC_{\parallel})\) of the
complex inside of the \(T_{sc}\)-equivariant \(K\)-theory. 
Thus the Euler characteristics of the LHS of last formula can be computed within \(K_{T_{sc}}(\FHilb^{free})\) and here we can use
the analog Negut's theorem for the push-forward along the fibers of the projection

\begin{prop} For any rational function \(r(\calL_n)\) with coefficients rational functions of \(\calL_i\), \(i<n\),
  the K-theory push-forward is given by
  \[\pi_*(r(\calL_n))=\int \frac{r(z)}{(1-z^{-1})}\prod_{i=1}\zeta'(\calL_i/z)\frac{dz}{z}\]
  where the contour of integration separates the set \(\mathrm{Poles}(r(z))\bigcup \{0,\infty\}\) from the poles of the rest of the integrant.
\end{prop}

The  \(K\)-theory class of the complex \(\CH^{\rst}_{\rloc}(\calC_\parallel^{ev})\) is \(\prod_{1\le i< j\le n}(1-qt\calL_i/\calL_j)\).
Hence we can apply the
formula from the previous proposition iterative to obtain the iterated residue integral formula for the desired
link invariant:
\[\int\dots\int  \prod_i\frac{z_i^{b_i}(1+az_i^{-1})}{1-z^{-1}}\prod_{1\le i<j\le n}\zeta(\frac{z_i}{z_j})\frac{d z_1}{z_1}\dots\frac{d z_n}{z_n}.
\]
The final step of the proof is a delicate analysis of the iterated residue that was done in the work of Negut \cite{Negut15} in the
context of \(K\)-theory of the flag Hilbert scheme.

\end{document}